\documentclass[reqno]{amsart}     
\usepackage{amsmath,amssymb,amsaddr}    
\usepackage{enumerate}
 
\newtheorem{theorem}                   {Theorem} 
\newtheorem{thm}             [theorem] {Theorem} 
\newtheorem{lemma}           [theorem] {Lemma}  
   
\newtheorem{corollary}       [theorem] {Corollary}   
\newtheorem{cor}             [theorem] {Corollary}

\newtheorem{claim}           [theorem] {Claim}

\theoremstyle{remark}
\newtheorem{remark}          [theorem] {Remark}

\newcommand{\eps}{\varepsilon}
\newcommand{\cH}{\mathcal{H}}
\newcommand{\cL}{\mathcal{L}}
\newcommand{\EE}{\mathbb{E}}
\newcommand{\Nat}{\mathbb{N}}
\newcommand{\Bi}{\mathrm{Bi}}
\newcommand{\Var}{\mathrm{Var}}

\newcommand{\cond}{\; \middle\vert \;}
\newcommand{\event}[1]{\mathcal{A}_{#1}}

\newcommand{\disc}{{ G_j}}   
\newcommand{\lbr}{{ \mathcal{T}_*}}                 
\newcommand{\slbr}{{ \left|\mathcal{T}_*\right|}}
\newcommand{\ubr}{{ \mathcal{T}^*}}                
\newcommand{\subr}{{ \left|\mathcal{T}^*\right|}}
\newcommand{\eo}{{ \eta}}                              
\newcommand{\ubfs}{{ \mathrm{BFS}}}          
\newcommand{\lbfs}{{ \mathrm{BFS2}}}         
\newcommand{\elbr}{{ \eps_*}}                        
\newcommand{\compone}{{ C_{J_1}}}             
\newcommand{\comptwo}{{ C_{J_2}}}             
\newcommand{\comp}{{ C_{J}}}                      
\newcommand{\generationarb}[1]{{ \partial \comp({#1})}}  
\newcommand{\compsmooth}{{ C_{J_0}}}                         
\newcommand{\generationsmooth}[1]{{ \partial \compsmooth({#1})}} 
\newcommand{\scomptwo}{{ \left|C_{J_2}\right|}}                        
\newcommand{\boundary}{{ \partial \compone}}                  
\newcommand{\generation}[1]{{ \partial \compone({#1})}} 
\newcommand{\deggen}[2]{{ d_{#2}(#1)}}         
\newcommand{\nbrhdbrconst}[1]{{ r_{#1}}}        
\newcommand{\co}{{ \eta}}                                   
\newcommand{\starttime}[1]{{ i_0(#1)}}              
\newcommand{\smoothtime}[1]{{ i^*(#1)}}        
\newcommand{\smoothend}[1]{{ i_1(#1)}}          
\newcommand{\varstarttime}[1]{{ i^{\dagger}(#1)}} 
\newcommand{\stoptime}{{ i_1}}                         
\newcommand{\jumpsdeg}[3]{{ d_{#2,#3}^*(#1)}}  
\newcommand{\expprob}{{ \exp(-\Theta(n^{\delta/2}))}}
\newcommand{\expprobtwo}{{ \exp(-\Theta(n^{\delta/4}))}}
\newcommand{\smootherror}[1]{{ \gamma_{#1}}}        
\newcommand{\cconst}{{ c_0}}					
\newcommand{\cconstl}[1]{{ c_{#1}}}			
\newcommand{\annconst}{{ c}}                    
\newcommand{\annerror}[1]{{ C_{#1}}}     
\newcommand{\annexcess}{{ (1+\delta)}}   
\newcommand{\annalpha}{{ \xi}}                  
\newcommand{\annea}[1]{{ x_{#1}}}       
\newcommand{\anntime}{{ \tau}}        
\newcommand{\gensize}[1]{{ |\partial({#1})|}} 
\newcommand{\jgensize}[1]{{ |\partial C_J({#1})|}}
\newcommand{\lbrgensize}[1]{{ |\partial_*({#1})|}}
\newcommand{\gen}[1]{{ \partial({#1})}}                
\newcommand{\adistr}{{ \mathcal{Z}}}                   
\newcommand{\lcompsize}{{ \Lambda}}                  
\newcommand{\lcompconst}{{ \lambda}}                
\newcommand{\lcomps}{{ \cL}}                     
\newcommand{\death}{{ \mathcal{D}}}			
\newcommand{\dualbr}{{ \mathcal{T}_{\mathcal{D}}}}				
\newcommand{\sdualbr}{{ \left|\mathcal{T}_{\mathcal{D}}\right|}}				
\newcommand{\pdualbr}{{ p_{\mathcal{D}}}}				
\newcommand{\primaryStop}{{ \mathcal{E}}}				
\newcommand{\primaryWidth}{{ \mathcal{W}}}				
\newcommand{\ubrone}{{ \mathcal{T}^{J_1}}}				
\newcommand{\subrone}{{ \left|\mathcal{T}^{J_1}\right|}}
\newcommand{\ubrtwo}{{ \mathcal{T}^{J_2}}}				
\newcommand{\subrtwo}{{ \left|\mathcal{T}^{J_2}\right|}}	
\newcommand{\falseNeg}{{ \mathcal{F}}}

\begin{document}

\title[The size of the giant component in random hypergraphs]{The size of the giant component in random hypergraphs}
\thanks{The authors are supported by Austrian Science Fund (FWF): P26826\\
The second author is additionally supported by W1230, Doctoral Program ``Discrete Mathematics''.}

\author[O.~Cooley, M.~Kang and C.~Koch]{Oliver Cooley, Mihyun Kang and Christoph Koch}
\email{cooley@math.tugraz.at}
\email{kang@math.tugraz.at}
\email{ckoch@math.tugraz.at}
\address{Institute of Optimization and Discrete Mathematics,\\ Graz University of Technology, 8010 Graz, Austria}

\date{\today}

\begin{abstract}
The phase transition in the size of the giant component in random graphs is one of the most well-studied phenomena in random graph theory. For hypergraphs, there are many possible generalisations of the notion of a component, and for all but the simplest example, the phase transition phenomenon was first proved in~\cite{CoKaPe14}. In this paper we build on this and determine the asymptotic size of the unique giant component.
\end{abstract}

\maketitle

\section{Introduction and main results}

Random graph theory was founded by Erd\H{o}s and R\'enyi in a seminal series of papers from 1959-1968. One of the earliest, and perhaps the most important, result concerned the \emph{phase transition} in the size of the largest component - a very small change in the number of edges present in the random graph dramatically alters the size of the largest component giving birth to the \emph{giant component}~\cite{ErdRen60}. Over the years, this result has been refined and improved, and the properties of the largest component at or near the critical threshold are now very well understood.  With modern terminology, and incorporating the strengthenings due to Bollob\'as~\cite{Bollobas84} and {\L}uczak~\cite{Luczak90} we may state the result as follows.

Let $G(n,p)$ denote the random graph on $n$ vertices in which each pair of vertices forms an edge with probability $p$ independently. We consider the asymptotic properties of $G(n,p)$ as $n$ tends to infinity. By the phrase \emph{with high probability} (or \emph{whp}) we mean with probability tending to $1$ as $n$ tends to infinity.

\begin{theorem}[\cite{ErdRen60},\cite{Bollobas84},\cite{Luczak90}]\label{thm:graphcase}
Let $\eps = \eps(n)>0$ satisfy $\eps \rightarrow 0$ and $\eps^3 n \rightarrow \infty$, as $n\to \infty$.
\begin{enumerate}[(a)]
\item If $p=\frac{1-\eps}{n}$ then whp all components of $G(n,p)$ have size at most $O(\eps^{-2}\log (\eps^3n))$.
\item If $p=\frac{1+\eps}{n}$ then whp the size of the largest component of $G(n,p)$ is $(1\pm o(1))2\eps n$ while all other components have size at most $O(\eps^{-2}\log (\eps^3n))$.
\end{enumerate}
\end{theorem}

Considerably more than this is by now known, but in this paper our focus is on a generalisation of this result. While this graph result (and much more) has been known for several decades, for hypergraphs relatively little was known until recently.

For some integer $k\ge 2$ a $k$-uniform hypergraph consists of a set $V$ of vertices and a set $E\subset \binom{V}{k}$ of edges, each of which consists of $k$ vertices. (The case $k=2$ corresponds to a graph.) We consider the natural analogue of the $G(n,p)$ model:
Let $\cH^k(n,p)$ be a random $k$-uniform hypergraph with vertex set $V=[n]$ in which each $k$-tuple of vertices is an edge with probability $p$ independently.

Before we can state a generalisation of Theorem~\ref{thm:graphcase} we need to know what we mean by a component in hypergraphs. For $k\ge 3$ there is more than one possible definition.

Two $j$-sets of vertices (i.e. $j$-element subsets of the vertex set $V$) $J_1,J_2$ are said to be \emph{$j$-connected} if there is a sequence $E_1,\ldots,E_\ell$ of (hyper-)edges such that $J_1 \subseteq E_1$, $J_2\subseteq E_\ell$ and $|E_i \cap E_{i+1}| \ge j$ for each $i=1,\ldots,\ell-1$. This forms an equivalence relation on $j$-sets. A \emph{$j$-component} is an equivalence class of this relation, i.e. a maximal set of pairwise $j$-connected $j$-sets.

The case $j=1$ is sometimes known as vertex-connectivity. This case is by far the most studied, not necessarily because it is a more natural definition, but because it is usually substantially easier to understand and analyse.

In~\cite{CoKaPe14}, it was determined that for all $1\le j \le k-1$, the phase transition for the largest $j$-component in random hypergraphs occurs at the critical probability threshold of
$$
p_0=p_0(k,j) := \tfrac{1}{\binom{k}{j}-1}\tfrac{1}{\binom{n}{k-j}}.
$$
However, in that paper the size of the largest $j$-component after the phase transition was not determined. In this paper we extend that result to prove its asymptotic size, and also show that it is unique in the sense that the size of the second-largest $j$-component is of smaller order. Once we know that it is unique, we refer to it as the \emph{giant component}.

\begin{thm}\label{thm:main}
Let $1 \le j \le k-1$ and let $\eps=\eps(n)>0$ satisfy $\eps\rightarrow 0$, $\eps^3 n^j\rightarrow \infty$ and $\eps^2 n^{1-2\delta} \rightarrow \infty$, as $n\to\infty$, for some constant $\delta >0$.
\begin{enumerate}[(a)]
\item \label{thm:main:subcrit} If $p=(1-\eps)p_0$ then whp all $j$-components of $\cH^k(n,p)$ have size at most $O(\eps^{-2}\log n)$.
\item \label{thm:main:supercrit} If $p=(1+\eps)p_0$ then whp the size of the largest $j$-component of $\cH^k(n,p)$ is $(1\pm o(1))\frac{2\eps}{\binom{k}{j}-1}\binom{n}{j}$ while all other $j$-components have size at most $o(\eps n^j)$.
\end{enumerate}
\end{thm}

While preparing this paper, we learnt that independently Lu and Peng~\cite{LP14} have also claimed to have a proof of a similar result, although for the range when $\eps$ is bounded from below by a constant (i.e.\ $p$ is much further from the critical value).

We note that when $j=1$ the conditions on $\eps$ are simply $n^{-1/3}\ll \eps \ll 1$. This provides the best possible range for $\eps$. However, for larger $j$, the third condition takes over from the second and we require $n^{-1/2+\delta}\ll \eps \ll 1$. This condition arises from our proof method. We discuss the critical window in more detail in Section~\ref{sec:conclusion}.

The case $k=2$ (then automatically $j=1)$ is simply Theorem~\ref{thm:graphcase}. The case $j=1$ for general $k$ was already proved by Schmidt-Pruzan and Shamir~\cite{SPS85} and indeed much more is known for that case, as will also be described in Section~\ref{sec:conclusion}.

Our proof works for all $j,k\in \Nat$ with $1\le j <k$ (i.e. all permissible pairs $j,k$). Indeed, for $j=1$ this paper provides a new, short proof of the result. For general $j$, in the \emph{supercritical regime} (when $p=(1+\eps)p_0$), we need an additional lemma (smooth boundary lemma, Lemma~\ref{lem:uniformboundary}), which is the main original contribution of this paper (the rest of the argument is in essence very similar to a recent proof of the graph case by Bollob\'as and Riordan~\cite{BR12b}, though slightly more complex in full generality). We will prove Theorem~\ref{thm:main} using a breadth-first search algorithm to explore the component containing an initial $j$-set. We refer to the collection of $j$-sets at fixed distance from the initial $j$-set as a \emph{generation}. Roughly speaking, the smooth boundary lemma says that during the (supercritical) search process, most generations are ``smooth'' in the sense that any set $L$ of size at most $j-1$ lies in approximately the ``right'' number of $j$-sets of the generation.

In order to state the lemma, we need a few additional definitions. For this introduction we will be deliberately vague about these definitions and the statement of the smooth boundary lemma, which appears in this section as Lemma~\ref{lem:crudeuniformboundary}. All definitions and exact form of the smooth boundary lemma (Lemma~\ref{lem:uniformboundary}) will be restated more explicitly in Section~\ref{sec:supercrit}.

For any given $j$-set $J$ we explore its component $\comp$ via a breadth-first search algorithm. Let $\generationarb{i}$ denote the $i$-th generation of this process, which we sometimes refer to as the \emph{boundary} of the component after $i$ generations. For any $0<\ell<j$ let $\starttime{\ell}$ be the first generation $i$ for which $\generationarb{i}$ is significantly larger than $n^\ell$ and for an $\ell$-set $L$ let $d_L(\generationarb{i})$ be the number of $j$-sets of $\generationarb{i}$ that contain $L$.

Let $\stoptime$ denote the generation at which the search process hits one of three stopping conditions, which will be stated explicitly later (see Section~\ref{sec:supercrit:secondmoment}).

\begin{lemma}[Smooth boundary lemma -- simplified form]\label{lem:crudeuniformboundary}
Let $\eps,p$ be as in Theorem~\ref{thm:main}~\eqref{thm:main:supercrit}. Then with probability at least $1-\exp (-n^{\Theta(1)})$, for all $J,\ell,L,i$ such that
\begin{itemize}
\item $J$ is a $j$-set of vertices;
\item $0\le \ell \le j-1$;
\item $L$ is an $\ell$-set of vertices;
\item $\starttime{\ell}+\Theta(\log n) \le i \le \stoptime$
\end{itemize}
the following holds:
$$
d_L(\generationarb{i}) = (1\pm o(1))\frac{|\generationarb{i}|}{\binom{n}{j}}\binom{n}{j-\ell}.
$$
\end{lemma}

Lemma~\ref{lem:crudeuniformboundary} (or its more explicit form, Lemma~\ref{lem:uniformboundary}) is an interesting result in itself, giving strong information about the structure of reasonably large components. We believe that it will prove to be a useful tool applicable to many other results in the field of random hypergraphs.

\section{Overview and proof methods}\label{sec:overview}

In this section we will be informal about quantifications and use terms like ``large'' and ``many'' without properly defining them. The rigorous definitions can be found in the actual proofs.

\subsection{Intuition: Branching processes}

As is the case for graphs, there is a very simple heuristic argument which suggests at what threshold we expect the phase transition to occur and how large we expect the largest component to be. We consider exploring a $j$-component of $\cH^k(n,p)$ via a breadth-first process: We begin with one $j$-set, then find all edges containing that $j$-set, thus ``discovering'' any further $j$-sets that they contain, then from each of these new $j$-sets in turn we look for any more edges containing them and so on.

The first $j$-set we consider is contained in $\binom{n-j}{k-j}$ $k$-sets, all of which could potentially be edges. Later on when considering an active $j$-set we may already have queried some of the $k$-sets containing it. However, early in the process, $\binom{n}{k-j}$ is a good approximation (and certainly an upper bound) for the number of queries which we make from any $j$-set. Each of these queries results in an edge with probability $p$, and for each edge we generally discover $\binom{k}{j}-1$ new $j$-sets (it could be fewer if some of these $j$-sets were already discovered, but intuitively this should not happen often).

We may therefore approximate the search process by a branching process $\ubr$: Here we represent $j$-sets by individuals and for each individual the number of its children is given by a random variable with distribution $\Bi\big(\binom{n}{k-j},p\big)$ multiplied by $\binom{k}{j}-1 $.

The expected number of children of each individual is $\big(\binom{k}{j}-1\big)\binom{n}{k-j}p$. If this number is smaller than $1$, then the process always has a tendency to shrink and will therefore die out with probability $1$. This roughly corresponds to the initial $j$-set being in a small component. On the other hand, if the expected number of children is larger than $1$, the process has a tendency to grow and therefore has a certain positive probability of surviving indefinitely, corresponding to the initial $j$-set being in a large component.

The expected number of children is exactly $1$ when $p=p_0$, which is why we expect the threshold to be located there. Furthermore, for $p=(1+\eps)p_0$ the survival probability $\varrho$ is asymptotically $\frac{2\eps}{\binom{k}{j}-1}$ (this is calculated in Appendix~\ref{app:branchings}). This tells us that we should expect about $\varrho \binom{n}{j}$ of the $j$-sets of $\cH^k(n,p)$ to be contained in large components. Moreover, large components should intuitively merge quickly and thus form a unique giant component.

\subsection{Proof outline}
As is often the case in such theorems, one half of Theorem~\ref{thm:main} is easy to prove. Specifically, part~\eqref{thm:main:subcrit} can be quickly proved using the ideas implemented by Krivelevich and Sudakov~\cite{KrivSud13} for the graph case.

We now give an outline of the proof of Theorem~\ref{thm:main}~\eqref{thm:main:supercrit}. Much of this is similar to the argument for graphs given by Bollob\'as and Riordan~\cite{BR12b} -- we will highlight the points at which new ideas are required. From now on, whenever we refer to a ``component'' we mean a $j$-component.

The main difficulty is to calculate the number $X$ of $j$-sets in large components; if we can prove that $X$ is well-concentrated, then we can show that almost all of these $j$-sets lie in one large component using a standard sprinkling argument.

The first issue that we encounter is, if we find an edge, how many new $j$-sets do we discover? It could be as many as $\binom{k}{j}-1$, but on the other hand some of these $j$-sets may already have been discovered, so perhaps only one of these is genuinely new. (Any $k$-set which would give no new $j$-set should not be queried at all.) This is important for two reasons: Firstly, it is important for the survival probability of the branching process approximation; and secondly, it may have a significant impact on the size of the component we discover.

Early on in the breadth-first process (since a very small proportion of $j$-sets have already been discovered) we should discover $\binom{k}{j}-1$ new $j$-sets for almost every edge. For an upper coupling, we will simply assume that we discover $\binom{k}{j}-1$ new $j$-sets for each edge, and couple this process with the branching process $\ubr$. In order to construct a lower coupling we have to be a bit more careful; we only query a $k$-set if it contains $\binom{k}{j}-1$ previously undiscovered $j$-sets, thus ensuring that we can define a lower coupling with a branching process $\lbr$ which has a structure similar to $\ubr$. The process $\lbr$ will be formally introduced in Section~\ref{sec:prelim}. It follows from the bounded degree lemma in~\cite{CoKaPe14} (restated in this paper as Lemma~\ref{lem:crudemaxdeg}) that these two couplings have essentially the same behaviour. For this overview we therefore consider only $\ubr$.

The probability that a $j$-set lies in a large component, and therefore contributes to $X$, is approximately the survival probability of the branching process $\ubr$, which as we will see is approximately $\tfrac{2\eps}{\binom{k}{j}-1}$. Thus we have $\EE (X)\sim \tfrac{2\eps}{\binom{k}{j}-1}\binom{n}{j}$. For the second moment we need to consider the probability that two distinct $j$-sets are both in large components. As before we grow a (partial) component $\compone$ from a $j$-set $J_1$ until one of the following three stopping conditions is reached.

\begin{enumerate}[(i)]
\item The component is fully explored;
\item We have discovered ``many'' $j$-sets;
\item We have ``fairly many'' $j$-sets that are currently active (i.e. are in the boundary $\boundary$).
\end{enumerate}

Since we are interested in the probability that both $j$-sets lie in large components we may assume that we do not stop due to condition (i). Next, note that, if stopping condition (iii) is applied, then with high probability $J_1$ lies in a large component. This is not hard to prove using the branching process approximation - if we have fairly many active individuals, then it is highly probable that the branching process will survive. Hence stopping conditions (ii) or (iii) are essentially only applied if the component of $J_1$ is large.  This happens with probability roughly $\frac{2\eps}{\binom{k}{j}-1}$.

We then delete all of the $j$-sets contained in $\compone$ from $\cH^k(n,p)$ and begin growing a component $\comptwo$ from the second $j$-set $J_2$ (assuming $J_2$ itself has not been deleted), where any $k$-set containing a deleted $j$-set can now no longer be queried. Deleting these $j$-sets ensures that the new search process is independent of the old process, albeit in a restricted hypergraph. Furthermore, it still follows from the bounded degree lemma that $\lbr$ will be a lower coupling. Once again, we stop the process if $\comptwo$ is fully explored or it becomes large, and the probability that it becomes large is approximately $\tfrac{2\eps}{\binom{k}{j}-1}$.

However, since we deleted some $j$-sets, it might happen that the search process for $J_2$ stays small even though the component of $J_2$ in $H^k(n,p)$ is large.  In which case there is an edge in $H^k(n,p)$ containing a $j$-set of $\compone$ which was active when we deleted it and a $j$-set of $\comptwo$, i.e. an edge containing $j$-sets from both $\boundary$ and $\comptwo$. We would like to show that the expected number of such edges is $o(1)$, or equivalently, that the number of $k$-sets containing two $j$-sets as above is $o(1/p)$, and thus with high probability no such edge exists.

This is the point at which the proof requires new ideas not needed in the graph case. For it is not enough simply to count the number of pairs of $j$-sets, one from $\boundary$ and one from $\comptwo$, for the following reason: Given two $j$-sets, how many $k$-sets contain both of them? The answer is heavily dependent on the size of their intersection. Increasing the size of the intersection by just one may lead to an additional factor of $n$ in the number of such $k$-sets.

We therefore need to know that $\comptwo$ and $\boundary$ behave approximately as expected with respect to the size of intersections of $j$-sets chosen one from each. For this we prove the \emph{smooth boundary lemma} (Lemma~\ref{lem:crudeuniformboundary}, or rather its more precise form Lemma~\ref{lem:uniformboundary}). It states that, with (exponentially) high probability, every set $L$ of size $1\le \ell< j$ lies in approximately the ``right'' number of $j$-sets of $\boundary$.

This enables us to complete the proof by considering, for each $j$-set $J$ of $\comptwo$, the number of $j$-sets of $\boundary$ which intersect $J$ in some subset $L$, and thus calculate the total number of $k$-sets which we have not queried because of deletions. This shows that with high probability we have not missed any edges from $\comptwo$, and thus the probability that $J_1$ and $J_2$ both lie in large components is approximately $4\eps^2/\big(\binom{k}{j}-1\big)^2$, which multiplied by the number of pairs $(J_1,J_2)$ shows that the second moment is small enough to apply Chebyshev's inequality and deduce that $X$ is concentrated around its expectation.

Note that Lemma~\ref{lem:crudeuniformboundary} (respectively Lemma~\ref{lem:uniformboundary}) is trivial in the case $j=1$. In this case the proof of the main result therefore becomes substantially shorter. This suggests a concrete mathematical reason why the case of vertex-connectivity is genuinely easier than the more general case and not simply easier to visualise.

\section{Notation and setup}\label{sec:prelim}

We fix $j,k\in\Nat$ satisfying $1\le j< k$ for the remainder of the paper. Throughout the paper we omit floors and ceilings when they do not significantly affect the argument. We use $\log$ to denote the natural logarithm (i.e. base $e$).

Given a hypergraph $\cH$ and a set $L$ of vertices of $\cH$, the \emph{degree} of $L$ in $\cH$, denoted $d_L (\cH)$, is the number of edges of $\cH$ which contain $L$. For a natural number $\ell$, the \emph{maximum $\ell$-degree} of $\cH$, denoted $\Delta_\ell (\cH)$, is the maximum of $d_L(\cH)$ over all sets $L$ of $\ell$ vertices of $\cH$. When $\ell=0$, this is simply the number of edges of $\cH$.

All asymptotics in the paper will be as $n\rightarrow \infty$. In particular, we use the phrase \emph{with high probability}, or \emph{whp}, to mean with probability tending to $1$ as $n \rightarrow \infty$. We also use the notation $f(n) \sim g(n)$ to mean $f(n)= (1\pm o(1))g(n)$. By the notation $x\ll y$ we mean that $x=o(y).$

Given two random variables $X,Y$, we say that \emph{$X$ is stochastically dominated by $Y$} if $\Pr(X\ge z)\le \Pr(Y\ge z)$ for all $z\in \mathbb{R}$.

As may be apparent from the introduction, the constant $\binom{k}{j}-1$ will be an important one, appearing many times during the proof. For convenience, we define
\[
\cconst:=\binom{k}{j}-1.
\]
 We fix a constant $\delta$ satisfying $0 < \delta < 1/6$, and think of it as an arbitrarily small constant -- in general our results become stronger for smaller $\delta$.

We will have various further parameters throughout the paper satisfying the following hierarchies:
\[
n^{-j/3},n^{-1/2+\delta} \ll \lcompconst \ll \elbr \ll \eps \ll 1
\]
and
\[
\elbr, n^{-\delta/24} \ll \eo \ll \smootherror{0}/(\log n) \ll 1/\log n.
\]
We further define $\smootherror{\ell}=8^{\ell} \smootherror{0}$ for $\ell=1,\ldots,j-1$.

Note in particular that for any $\eps$ satisfying the conditions of Theorem~\ref{thm:main} we can choose the remaining parameters such that this hierarchy is satisfied.

We will explore $j$-components in $\cH^k(n,p)$ via a breadth-first search algorithm, i.e. we begin with a $j$-set and query all $k$-sets which contain it to determine whether they form edges. For any that do, the further $j$-sets they contain are neighbours of our starting $j$-set, and for each of these in turn we query $k$-sets containing them to discover whether they form an edge, and so on.

During this process we denote a $j$-set as:
\begin{itemize}
\item  \textbf{neutral} if it has not yet been visited by the search process;
\item  \textbf{active} if it has been visited, but not yet fully queried;
\item  \textbf{explored} if it has been fully queried.
\end{itemize}
We refer to \emph{discovered} $j$-sets to mean $j$-sets that are either active or explored, but not neutral.
We will consider a standard breadth-first search algorithm:
\begin{itemize}
\item $\ubfs(J)$: From an initial $j$-set $J$ we query any previously unqueried $k$-set containing it which also contains a still-neutral $j$-set. If the set of active $j$-sets becomes empty we terminate the algorithm.
\end{itemize}
If the context clarifies the initially active $j$-set we usually write $\ubfs$ instead of $\ubfs(J)$. 

We will want to approximate this search processes by branching processes:

\begin{itemize}
\item $\ubr$ is a branching process in which each individual (which represents a $j$-set) has $\cconst \cdot\Bi\big(\binom{n}{k-j},p\big)$ children independently.
\item $\lbr$ is a branching process in which each individual (which represents a $j$-set) has $\cconst \cdot\Bi\big((1-\elbr)\binom{n}{k-j},p\big)$ children independently.
\end{itemize}
By the notation $c\cdot X$, for a constant $c$ and a random variable $X$, we mean a random variable $Y$ with distribution given by $\Pr (Y = ci) = \Pr (X=i)$, for any $i$. (Alternatively $c\cdot X$ may be considered as consisting of $c$ \emph{identical} copies of $X$ -- note that it does \emph{not} consist of $c$ \emph{independent} copies of $X$.)

It is immediately clear that $\ubr$ forms an upper coupling for $\ubfs(J)$. That $\lbr$ is a lower coupling (whp) is less obvious, but will be proved later (see Lemma~\ref{lem:approximations}). 

\begin{remark}
It is for this reason that we need $\elbr \ll \eps$ -- then in the lower coupling $\lbr$ the expected number of children is still approximately $1+\eps$, i.e. the lower coupling and upper coupling are still very similar.
\end{remark}

At several points in the argument we will use the following form of the Chernoff bound for sums of indicator random variables (see e.g.~\cite[Theorem~2.1]{JLRbook}).
\begin{theorem}\label{thm:chernoff}
Let $X$ be the sum of finitely many i.i.d.\ Bernoulli random variables. Then for any $\zeta \ge 0$,
\begin{align}
   \Pr\left[X\ge \EE(X)+\zeta\right] & \le \exp\left(-\frac{\zeta^2}{2\left(\EE(X)+\zeta/3\right)}\right) \label{eq:upperchernoff}\\
   \Pr\left[X\le \EE(X)-\zeta\right] & \le \exp\left(-\frac{\zeta^2}{2\, \EE(X)}\right).\label{eq:lowerchernoff}
\end{align}
\end{theorem}

\section{Subcritical phase}

In this section we prove Theorem~\ref{thm:main}\eqref{thm:main:subcrit}. The proof idea here is the same as that of the subcritical graph case as proved by Krivelevich and Sudakov~\cite{KrivSud13}. We observe that a component of size $m$ must have at least $m/\cconst$ edges, which were found within an interval of length at most $m\binom{n}{k-j}$ of the search process. Let us consider the probability that an interval of this length contains so many edges.
\begin{align*}
\Pr \left(\Bi \left( m\binom{n}{k-j},(1-\eps)p_0 \right) \ge m/\cconst \right) & \stackrel{\mbox{{\tiny Thm \ref{thm:chernoff}}}}{\le} \exp\left(-\frac{\eps^2 m^2/\cconst^2}{2\left((1-\eps)m/\cconst + \eps m/3\cconst\right)}\right)\\
& \le \exp \left( -\frac{\eps^2 m}{2\cconst} \right).
\end{align*}
If $m \ge 3\cconst k\eps^{-2}\log n $, then this probability is at most $n^{-3k/2}=o(n^{-k})$, and therefore we may take a union bound over all possible starting points for the interval, of which there are at most $\binom{n}{k} < n^k$, and still have a probability of $o(1)$. In other words, with high probability no such interval exists, and therefore no component of size $m\ge 3\cconst k\eps^{-2}\log n$ exists. Note that we were not concerned about optimising this bound on $m$.\qed

\section{Supercritical phase}\label{sec:supercrit}

In this section we prove Theorem~\ref{thm:main}~\eqref{thm:main:supercrit}, which is substantially harder. We first gather some basic, preliminary results which we will use several times during the argument.

\subsection{Preliminaries}

Let $\disc = \disc (t;J)$ denote the $j$-uniform hypergraph with vertex set $V=[n]$ whose edges are those $j$-sets which have been discovered by $\ubfs(J)$ up to time $t$, i.e.\ having made precisely $t$ queries so far. We will omit the arguments $t$ and $J$ if they are clear from the context.

The following lemma (Lemma~12 from~\cite{CoKaPe14}) will play a critical role; it says that early on in the search process (in particular in the time period we are interested in) the edges of $\disc$ are nicely distributed in the sense that no set of vertices is contained in too many $j$-sets of $\disc$.

\begin{lemma}[Bounded degree lemma~\cite{CoKaPe14}]\label{lem:crudemaxdeg}
Let $n^{-1+\delta} \ll \alpha \ll \eps \ll 1$. Then there exists a constant $C=C(k,j)$ such that we have with probability at least $1-\expprob$ for any $J,\ell,t$ satisfying
\begin{itemize}
\item $J$ is a $j$-set of vertices;
\item $0 \le \ell \le j-1$;
\item $t\le \alpha n^k$;
\end{itemize}
we have
\[
 \Delta_\ell (\disc(t;J)) \leq C \alpha n^{j-\ell}.
\]
\end{lemma}

As indicated in Section~\ref{sec:overview}, the main additional difficulty for the hypergraph case is the smooth boundary lemma (Lemma~\ref{lem:uniformboundary}), which may be seen as a significantly stronger version of Lemma~\ref{lem:crudemaxdeg}. However, the proof will make fundamental use of Lemma~\ref{lem:crudemaxdeg} and therefore does not render it superfluous.

As indicated previously, we aim to use $\lbr$ and $\ubr$ as lower and upper couplings on $\ubfs(J)$ for some initial $j$-set $J$. Let us describe more precisely what we mean by this. In our search process we will certainly always make at most $\binom{n}{k-j}$ queries from any $j$-set. If the actual number is $x\le \binom{n}{k-j}$, then we identify these with the first $x$ queries from an individual of $\ubr$ (which we view as a subtree of the infinite $\binom{n}{k-j}$-ary tree in which each edge is present with probability $p$, and we consider the subtree containing the root). The remaining $\binom{n}{k-j}-x$ queries in $\ubr$ are in effect ``dummy'' queries which do not exist in the search process, but making extra queries is permissible for an upper bound. Thus if we are at time $t$ in the search process, then $\ubr$ may have made more than $t$ queries. However, since we will generally be considering \emph{generations} of the search process rather than the exact time, this difference does not affect anything.

Similarly we can couple the search process with $\lbr$ by ignoring some queries which the search process makes, and considering only those queries which would give $\cconst$ new $j$-sets, and of these consider only the first $(1-\elbr)\binom{n}{k-j}$. Of course, this requires that there are at least this many such queries in the search process, which we will prove using Lemma~\ref{lem:crudemaxdeg}.

We will denote this coupling of processes (when it holds) by $\lbr \prec \ubfs(J) \prec \ubr$.

\begin{lemma}\label{lem:approximations}
With probability at least $1-\expprob$, the process $\ubfs(J)$ satisfies the following properties:
\begin{enumerate}[(A)]
\item \label{prop:edges} For every time $0\le t \le \binom{n}{k}$, the number of edges which we have found up to time $t$ is
\[\begin{cases}
(1\pm \eo)p_0t & \mbox{if } p_0t \ge n^{\delta}\\
\le (1+\eo)n^{\delta} & \mbox{otherwise.}
\end{cases}\]
\item \label{prop:crudemaxdeg} Given $n^{-1+\delta} \ll \alpha \ll \eps \ll 1$, after every $t\le \alpha n^k$ queries we have
$$
\Delta_\ell (\disc) \le C\alpha n^{j-\ell}
$$
 for all $ 1\le \ell \le j-1$ and for some constant $C=C(k,j)$.
\item \label{prop:coupling} If $\alpha \ll \elbr$, then for every $t \le \alpha n^k$, at time $t$ we have $\lbr \prec \ubfs\prec \ubr$. 
\end{enumerate}
\end{lemma}

\begin{proof}
Property~\eqref{prop:edges}: Let $e_t$ denote the number of edges found by time $t$. Note that $e_t$ has distribution $\Bi (t,p)$. Let us first consider the case $t\ge n^{\delta}/p_0$. Then by Theorem~\ref{thm:chernoff}
\[
\Pr \left(e_t \neq  (1\pm\eo) p_0 t\right) \leq 2\exp \left( -\frac{\eo^2 n^{2\delta}}{3 n^{\delta}} \right)  \le 2\exp (-n^{\delta/2}).
\]
Similarly if $t < n^{\delta}/p_0$ we have
\[
\Pr \left(e_t> (1+\eo) n^{\delta}\right) \le \Pr \left(\Bi \left(n^{\delta}/p_0, p_0\right)> (1+\eo)n^{\delta}\right) \le \exp (-n^{\delta/2}).
\]
Finally we apply a union bound over all times $t$ to bound the probability that Property~\eqref{prop:edges} is not satisfied by
\[
2\binom{n}{k} \exp (-n^{\delta/2}) \le \exp (-n^{\delta/2}/2) = \expprob
\]
as required.\\
Property~\eqref{prop:crudemaxdeg}: Follows directly from Lemma~\ref{lem:crudemaxdeg}.\\
Property~\eqref{prop:coupling}: The upper coupling is immediate from the definitions. The lower coupling follows from Property~\eqref{prop:crudemaxdeg}. More precisely, from any $j$-set $J$ we can bound the number of discovered $j$-sets which intersect $J$ in $\ell$ vertices from above by $\binom{j}{\ell}\Delta_\ell (\disc)$. For each such discovered $j$-set, the number of $k$-sets containing its union with $J$ is at most $\binom{n}{k-2j+\ell}$. Thus if Property~\eqref{prop:crudemaxdeg} holds then the number of queries from $J$ which would give fewer than $\cconst$ new $j$-sets is at most
\begin{align*}
\sum_{\ell=0}^{j-1} \binom{j}{\ell} \Delta_\ell (\disc) \binom{n}{k-2j+\ell} & = \sum_{\ell=0}^{j-1} O(\alpha n^{j-\ell} n^{k-2j+\ell})\\
& = O(\alpha n^{k-j}).
\end{align*}
Thus the number of $k$-sets which may still be queried from $J$ and which would give $\cconst$ new $j$-sets is at least
\[
\binom{n-j}{k-j}-O(\alpha n^{k-j}) \ge (1-\elbr)\binom{n}{k-j} 
\]
where the inequality holds since $1/n,\alpha \ll \elbr$.
\end{proof}

\begin{remark}
Note that Property~\eqref{prop:crudemaxdeg} and the lower coupling in Property~\eqref{prop:coupling} only hold early in the process. For the rest of this proof we will always assume that we are at a time $t\le \alpha n^k$ in the search process, for some $\alpha\ll \elbr$\,. Formally, we introduce a stopping condition, that we stop our search process at time $\alpha n^k$ if we haven't already, where we will have $\alpha = \Theta(\lcompconst) \ll \elbr$. However, we will have other stopping conditions and it will follow from those that with very high probability we will never reach time $\alpha n^k$ (see Lemma~\ref{rem:stopearly}).
\end{remark}

\begin{remark}
Throughout the paper we will have various Claims and Lemmas stating that a certain good event holds with very high probability, generally $\expprob$ (though later also $\expprobtwo$ appears). Without explicitly stating so, we will subsequently assume that the good event always holds. More formally, we introduce a new stopping condition for each lemma, and terminate the process if the corresponding good event does not hold. By a union bound over all bad events, as long as there are not too many of them, with very high probability no such stopping condition is ever invoked (note that $P(n)\cdot\expprob=\expprob$ for any polynomial $P$).
\end{remark}

We prove one more preliminary result which says that the expansion of the search process is approximately as fast as we expect (once the boundary becomes large).  In order to state the result, we first give some general notation: Consider the search process $\ubfs(J)$. Then we write $C_J=C_J(i)$ for the (partial) component that we have ``currently'' discovered, i.e.\ up to the $i$-th generation. If the meaning of ``currently'' is clear we sometimes drop the argument $i$. Similarly we use $\partial C_J=\partial C_J(i)$ to denote the $i$-th generation of the process.

\begin{lemma}[Bounded expansion]\label{lem:expansionbound}
Suppose that for some initial $j$-set $J$, at generation $i$, $\lbr \prec \ubfs(J) \prec \ubr$.
Conditioned on $\jgensize{i}=x_i \in \Nat$, with probability at least $1-\expprob$
\[\begin{cases}
\jgensize{i+1} = (1 \pm 2\elbr)(1+\eps)x_i & \mbox{if } x_i \ge n^{1-\delta}\\
\jgensize{i+1} \le 2\max (x_i,n^{\delta}) & \mbox{otherwise.}
\end{cases}
\]
\end{lemma}
Note that the second half of the statement gives us only a much weaker upper bound, but applies to smaller $x_i$. 

\begin{proof}
We prove the upper bound for the first half of the statement here. The other cases are very similar and are given in Appendix~\ref{app:proofs} for completeness.

Note that since $\ubr$ is an upper coupling, conditioned on $\jgensize{i}=x_i$, the size of the next generation $\jgensize{i+1}$ is stochastically dominated by a random variable $Y_{i+1}$ with distribution $\cconst \cdot \Bi \big(x_i \binom{n}{k-j},p\big)$. We have 
\[
\EE(Y_{i+1}) =\cconst x_i \binom{n}{k-j}p = (1+\eps)x_i.
\]
Furthermore, by the Chernoff bound (Theorem~\ref{thm:chernoff}) we have
\begin{align*}
\Pr (\jgensize{i+1}\ge (1+2\elbr)(1+\eps)x_i) & \le \Pr(Y_{i+1}\ge (1+2\elbr)(1+\eps)x_i)\\
& = \Pr \left( \tfrac{Y_{i+1}}{\cconst}\ge (1+2\elbr)\EE\left(\tfrac{Y_{i+1}}{\cconst}\right)\right)\\
& \le  \exp \left(\frac{-(2\elbr)^2 \EE(Y_{i+1})}{3\cconst}\right)\\
& = \exp (-\Theta (\elbr^2 x_i))\\
& \le \exp (-\Theta (n^{-1+2\delta} n^{1-\delta}))\\
& \le \expprob
\end{align*}
where for the penultimate inequality we used the fact that $\elbr \ge n^{-1/2+\delta}$.
This concludes the proof of the upper bound for the first half of the statement.
\end{proof}
We have the following corollary. Let $B_i$ be the event
\begin{align*}
B_i := & \left\{\jgensize{i} \ge n^{1-\delta} \wedge \jgensize{i+1} \neq (1 \pm 2\elbr)(1+\eps)\jgensize{i}\right\} 
\\& 
\vee \left\{\jgensize{i+1} > 2\max (x_i,n^{\delta}) \right\}.
\end{align*}

\begin{cor}\label{cor:expansionbound}
With probability at least $1- \expprob$, no $B_i$ occurs for any $i$ such that $\lbr \prec \ubfs(J) \prec \ubr$.
\end{cor}

\begin{proof}
For each $i$ and for each $x_i \in \mathbb{N}$, let $A_i(x_i):=\{\jgensize{i}=x_i\}$. Then
\begin{align*}
\Pr(B_i) = & \sum_{x_i} \Pr\left(B_i\cond A_i(x_i)\right)\Pr(A_i(x_i))\\
\le & \sum_{x_i}\Pr\left(\jgensize{i+1} > 2\max (x_i,n^{\delta})\cond A_i(x_i)\right)\Pr(A_i(x_i))\\
& +   \sum_{x_i \ge n^{1-\delta}}\Pr\left( \jgensize{i+1} \neq (1 \pm 2\elbr)(1+\eps)x_i\cond A_i(x_i)\right)\Pr(A_i(x_i))\\
\le & \sum_{x_i}\expprob \Pr(A_i(x_i)) 
 + \sum_{x_i \ge n^{1-\delta}}\expprob \Pr(A_i(x_i))\\
= & \, 2\, \expprob
\end{align*}
and a union bound over all choices of $i$ (of which there are certainly less than $n^k$) gives the desired result.
\end{proof}

Note that since $\elbr \ll \eps$, this lemma implies that in the time range that we are usually interested in (i.e. when $x_i \ge n^{1-\delta}$), with high probability the size of the boundary never decreases. This is useful since we later have various concentration results which always require the boundary to have a certain minimum size. This lemma ensures that these conditions will always remain valid once they are attained.

\begin{remark} 
The pattern of Lemma~\ref{lem:expansionbound} and Corollary~\ref{cor:expansionbound} will be repeated several times; we prove that a certain event holds with high probability, where the event is dependent on a random variable (i.e. the size of the boundary) but the probability is not. With this uniform probability we may always assume that the good event holds (if necessary taking a union bound, for example over different generations). From now on, we will only give the lemma and use the corresponding corollary implicitly without stating it.
\end{remark}

\subsection{Branching processes}

We will also need some results on branching processes. The arguments here are similar to standard and well-known arguments, but the actual processes may have an unfamiliar distribution, so for completeness we give the full arguments in Appendix~\ref{app:branchings}.

\subsubsection{Survival probability}\label{sec:survprob}

Consider the branching process $\ubr$, in which the number of children has distribution $\cconst \cdot \Bi \big(\binom{n}{k-j},p\big)$. For the supercritical case we have $p=(1+\eps)p_0$. By setting up a suitable recursion, we may deduce that the survival probability $\varrho$ of this process is the unique positive solution to the equation
\[
(1-\varrho)^\cconst = \sum_{i=0}^{\cconst\binom{n}{k-j}} \Pr \left(\Bi \left(\cconst\binom{n}{k-j},p\right)=i\right) (1-\varrho)^{\cconst i}
\]
which is asymptotically
\begin{equation}\label{survival1}
\varrho \sim 2\eps /\cconst
\end{equation}
(see Appendix~\ref{app:branchings}). Likewise the lower coupling process $\lbr$, in which the number of children has distribution $\cconst \cdot \Bi \big((1-\elbr)\binom{n}{k-j},p\big)$, has survival probability $\varrho_*$ with
\begin{equation}\label{survivalSub1}
\varrho_* \sim 2\eps /\cconst.
\end{equation}

\subsubsection{Dual process}

When we explore a component and the lower coupling $\lbr$ survives indefinitely we can conclude that the component must be large, since the search process certainly survives until $\lbr$ is no longer a lower coupling. However if the upper coupling $\ubr$ dies out we need some information on the total size of all its generations in order to decide whether this implies that the component must be small. 

Recall that $\ubr$ has the offspring distribution
$$\cconst \cdot \Bi\left(\binom{n}{k-j},p\right)$$
where $p=(1+\eps)p_o$. We denote  the event that $\ubr$ dies out by $\death$ and condition on it. This defines a conditional branching process $\dualbr,$ called the \emph{dual process}, with offspring distribution
$$\cconst \cdot \Bi\left(\binom{n}{k-j},\pdualbr\right),$$
where
\[
\pdualbr =(1-\eps + o(\eps))p_0
\]
(see Appendix~\ref{app:branchings} for a proof).
Thus the expected number of children of any individual in $\dualbr$ is given by
$$\cconst \binom{n}{k-j}\pdualbr=1-\eps+o(\eps)<1$$
and hence in particular the dual process is a subcritical branching process. Therefore we can give an asymptotic estimate on the expected total size of $\dualbr$ by standard techniques (see \cite{Harris63}):
\begin{align}\label{expDual}
\EE\left(\sdualbr\right)=\sum\limits_{i=0}^{\infty}\left(1-\eps+o(\eps)\right)^i=\frac{1}{1-(1-\eps+o(\eps))}\sim\eps^{-1}.
\end{align}
Consequently we can bound the probability of the process $\ubr$ being larger than some $\lcompsize=\lcompsize(n)$ by conditioning on $\death$ and applying Markov's Inequality
\begin{align}
\Pr\left(\subr\geq \lcompsize\right) & = \Pr \left( \neg \death\right)\cdot 1 + \Pr (\death)\cdot\Pr\left(\subr\geq \lcompsize\cond \death\right) \nonumber\\
&\leq \Pr\left(\neg \death\right) +1\cdot\Pr\left(\sdualbr\geq \lcompsize\right)\nonumber\\
& \stackrel{\eqref{expDual} }{\leq} \varrho + (1+o(1)) \eps^{-1}/\lcompsize\nonumber\\
& \stackrel{\eqref{survival1}}{\sim}  \frac{2\eps}{\cconst}\, ,\label{probUbrLarge}
\end{align}
as long as $\eps^2\lcompsize\to\infty.$

\subsection{First moment}

Let $X$ denote the number of $j$-sets in components of size at least
$$\lcompsize:=\lcompconst n^j$$
and observe that $\eps^2\lcompsize\to\infty$ is satisfied. The principle difficulty is in showing that $X$ is with high probability approximately as we would expect. We first calculate what this expectation is using the upper coupling $\ubr$ and the lower coupling $\lbr$ and obtain
\begin{align*}
\EE\left(X\right)\leq\binom{n}{j}\Pr\left(\subr\geq \lcompsize\right)\stackrel{\eqref{probUbrLarge}}{\sim} \frac{2\eps}{\cconst}\binom{n}{j}
\end{align*}
and similarly 
\begin{align*}
\EE\left(X\right)\geq\binom{n}{j}\Pr\left(\slbr=\infty\right)= \varrho_*\binom{n}{j}\stackrel{\eqref{survivalSub1}}{\sim} \frac{2\eps}{\cconst}\binom{n}{j}\, .
\end{align*}
Hence we have 
\begin{equation}
\EE (X) =(1\pm o(1))\frac{2\eps}{\cconst}\binom{n}{j}\, .
\label{firstMoment}
\end{equation}

\subsection{Second moment}\label{sec:supercrit:secondmoment}

Let $\lcomps$ denote the union of all components of size at least $\lcompsize$. In order to apply Chebyshev's inequality to prove that $X$ is concentrated around its expectation, we need to show that $\EE(X^2) \sim \EE(X)^2$. We may interpret $X^2$ as the number of ordered pairs of $j$-sets in large components (formally we may pick the same $j$-set twice in such a pair) and thus we can write its expectation as 
$$\EE\left(X^2\right)=\sum_{J_1\in\binom{V}{j},\, J_2\in\binom{V}{j}}\Pr\left(J_1,J_2\in\lcomps\right)\, . $$

Fix an arbitrary $j$-set $J_1$. We grow a component $\compone$ from $J_1$ using $\ubfs$.
We denote the upper coupling branching process for the exploration by $\ubrone$ (i.e. $\ubrone$ is a particular instance of $\ubr$). We continue to grow the component until one of the following three stopping conditions is reached:
\begin{enumerate}
\renewcommand{\theenumi}{S\arabic{enumi}}
\item \label{primaryStopCond1} The component $\compone$ is fully explored (i.e. no $j$-sets are still active);
\item \label{primaryStopCond2} The component $\compone$ has reached size $\lcompsize=\lcompconst n^j$;
\item \label{primaryStopCond3} The boundary of the component $\boundary$ has reached size $\lcompconst\lcompsize= \lcompconst^2 n^j$.
\end{enumerate}

A technical point on both the second and the third stopping conditions is that we check them only at the end of each generation, i.e. we do not stop in the middle of a generation. This is in contrast to the stopping procedure in the graph case in~\cite{BR12b} -- in the hypergraph case, stopping immediately would lead to significant technical difficulties due to the boundary being spread over two generations. Therefore our convention is much more convenient once paired with Corollary~\ref{cor:expansionbound}, which automatically implies that neither the boundary nor the whole component ends up being significantly bigger than the threshold for the stopping condition, as formulated in the following remark.

\begin{remark}\label{rem:overshoot}
With probability at least $1-\expprob$, when the stopping conditions are reached we have
\begin{align*}
|\boundary| & \le 2\lcompconst^2 n^j;\\
|\compone| & \le 2\lcompconst n^j.
\end{align*}
\end{remark}

Having grown $\compone$ in this way, let us first consider the contribution to $\EE (X^2)$ that comes from $j$-sets $J_2$ which lie inside $\compone$. By Remark~\ref{rem:overshoot}, there are at most $2\lcompsize$ such $j$-sets, and therefore the contribution is at most

\begin{align}
\sum_{J_1} \sum_{J_2 \in \compone} \Pr (J_1,J_2 \in \lcomps) & \le \binom{n}{j} 2\lcompsize \cdot O(\eps)\nonumber\\
& = O(\eps \lcompconst  n^{2j}) =  o(\eps^2 n^{2j}),\label{secondMomentintersecting}
\end{align}
Since $\EE(X^2)\ge \EE(X)^2 = \Theta (\eps^2 n^{2j})$, this contribution will be negligible.

Let us therefore assume that $J_2$ lies outside $\compone$ and fix it for the remainder of the proof. We delete all the $j$-sets of $\compone$ from $\cH^k(n,p)$ -- any $k$-sets containing them may now no longer be queried.

We now start a new $\ubfs$ process from $J_2$ and thus grow a component $\comptwo$ until one of the following two stopping conditions is reached:

\begin{enumerate}
\renewcommand{\theenumi}{T\arabic{enumi}}
\item \label{secondaryStopCond1} The component $\comptwo$ is fully explored;
\item \label{secondaryStopCond2} The component $\comptwo$ has reached size $\lcompsize =\lcompconst n^j$.
\end{enumerate}
Again, we only stop at the end of a generation.
The following lemma will be critical to the argument

\begin{lemma}\label{lem:stopearly}\label{rem:stopearly}
With probability at least $1-\expprob$, $\lbr$ is a lower coupling for the search processes of $\compone$ and $\comptwo$ until these stopping conditions are reached.
\end{lemma}

\begin{proof}
Note that when we stop exploring $\compone$, because of stopping condition~\eqref{primaryStopCond1} and Property~\eqref{prop:edges} we have made $O(\lcompconst n^k)$ queries in $\ubfs$. Therefore by Property~\eqref{prop:coupling}, $\lbr$ is a lower coupling for $\ubfs$ throughout the exploration of $\compone$.

We would like to say something similar about the second search process for $\comptwo$. However, we cannot simply apply Property~\eqref{prop:coupling} since it relied on Lemma~\ref{lem:crudemaxdeg}, which only applies to a search process within the complete hypergraph. The hypergraph in which we run the search process for $\comptwo$ is not complete, since the $j$-sets of $\compone$ have been deleted.

We solve this problem by considering an equivalent auxiliary search process in the complete hypergraph, but in which we reduce some edge probabilities to zero, specifically those $k$-sets containing $j$-sets of $\compone$. It is clear that this process is dominated by the usual process in the complete hypergraph, to which we can apply Lemma~\ref{lem:crudemaxdeg}. However, we still need to know that the number of queries in this auxiliary search process is $O(\lcompconst n^k)$, i.e. that we have not made too many additional dummy queries.

To prove this, we simply observe that the number of $j$-sets of $\compone$ is at most $2\lcompconst n^j$, by Remark~\ref{rem:overshoot}, and thus the number of $k$-sets containing such a $j$-set is at most $2\lcompconst n^j \binom{n}{k-j} = O(\lcompconst n^k)$. Therefore the total number of additional queries that we can possibly make is $O(\lcompconst n^k)$ as required.
\end{proof}

Lemma~\ref{lem:stopearly} tells us that the $\lbr$ is a lower coupling for both the first and the second search processes (which also means that $\ubr$ is a good approximation.) Therefore as calculated before for $\compone$, the probability that $\comptwo$ grows large is $(1\pm o(1))2\eps/\cconst$.

Let us now compare the reasons why the search processes stop with the events that the corresponding components are large. 
First, observe that if the exploration $\compone$ stopped because the component was fully explored, \eqref{primaryStopCond1}, then it never reached size $\lcompsize$ and therefore $J_1$ does not lie in a large component. Consequently it will not contribute to $X^2.$ Hence we are interested in the case when the exploration of the component of $J_1$ stops due to stopping condition \eqref{primaryStopCond2} or \eqref{primaryStopCond3} and we call this event $\primaryStop.$ From the previous observation it is immediate that 
 \begin{equation*}
 \left\{J_1\in\lcomps\right\} \implies \primaryStop
 \end{equation*}
  and thus
\begin{equation*}
\Pr\left(J_1,\,J_2\in\lcomps\right)\leq \Pr\left(\primaryStop\wedge\left\{J_2\in\lcomps\right\}\right)=\Pr\left(\primaryStop\right)\Pr\left(J_2\in\lcomps\cond \primaryStop\right)\,.
\end{equation*}
  Consequently we have 
\begin{align}
\EE\left(X^2\right)&\stackrel{\eqref{secondMomentintersecting}}{\leq} \binom{n}{j}\Pr\left(\primaryStop\right)\sum_{J_2\in\binom{V}{j}\setminus \compone}\Pr\left(J_2\in\lcomps\cond \primaryStop\right)+o\left(\eps^2n^{2j}\right)\nonumber\\
&\leq (1+o(1)) \binom{n}{j}\Pr\left(\primaryStop\right)\sum_{J_2\in\binom{V}{j}\setminus \compone}\Pr\left(J_2\in\lcomps\cond \primaryStop\right).
\label{secondMomentafterPrimaryStop}
\end{align}  

In order to give a suitable upper bound for $\Pr\left(\primaryStop\right)$ we have to analyse the upper coupling $\ubrone$ if we stop due to stopping condition \eqref{primaryStopCond2} or \eqref{primaryStopCond3}. For technical reasons we distinguish two cases in a way that might seem a little awkward: If stopping condition \eqref{primaryStopCond2} was implemented, then clearly $J_1$ does lie in a large component and therefore the \emph{total size} of $\ubrone$ will be at least $\lcompsize$ since we already reached that size when we stopped the exploration. This motivates a case distinction according to the following implication
\begin{equation}
\primaryStop \implies \left\{\subrone\geq \lcompsize\right\}\vee\left(\primaryStop\wedge \left\{\subrone<\lcompsize\right\}\right)\, .
\label{impl1}
\end{equation}
Let $\primaryWidth$ be the event that the generation of $\ubrone$ at which we stopped the exploration process is larger than $\lcompconst\lcompsize.$ Observe that
\begin{equation}
\primaryStop\wedge \left\{\subrone<\lcompsize\right\}\implies\primaryWidth\wedge\left\{\subrone<\lcompsize\right\}  \, ,
\label{impl2}
\end{equation}
since the event $\primaryStop$ can only hold (subject to $\left\{\subrone<\lcompsize\right\}$) if the boundary $\boundary$ of the explored component was at least of size $\lcompconst\lcompsize$ and hence the corresponding generation of the upper coupling must also have been large, i.e.\ $\primaryWidth$ needs to hold. Since 
\begin{equation}
\left\{\subrone<\lcompsize \right\}\implies \left\{\subrone<\infty\right\},
\label{impl3}
\end{equation}
 we want to consider the event that $\ubrone$ dies out after having had a large generation. Intuitively we would imagine that the chances of this happening are very small so let us make this intuition more precise. 

Assume that $\primaryWidth$ holds. Then there is a generation of $\ubrone$ with at least $\lcompconst \lcompsize$ $j$-sets in the boundary, and from each of these we start an independent copy of $\ubr$, all of which need to die out in order for $\ubrone$ to die out.
Thus we have
\begin{align}
\Pr\left(\subrone<\infty\cond \primaryWidth\right)\leq \Pr\left(\subr<\infty\right)^{\lcompconst\lcompsize} \leq \left(1-\frac{\eps}{2^k}\right)^{\lcompconst \lcompsize} \leq \exp\left(-\frac{\eps \lcompconst\lcompsize}{2^k}\right)=o(1)\, , \label{probStopWidth}
\end{align}
 since the survival probability of $\ubr$ is $(1\pm o(1))\frac{2\eps}{\binom{k}{j}-1} \geq \eps/2^k$. This implies
 \begin{align}
 \Pr\left(\primaryWidth\wedge \left\{\subrone<\infty\right\}\right)
& = \Pr(\primaryWidth) \Pr\left(\subrone < \infty \cond \primaryWidth\right)  \nonumber \\
&=\Pr\left(\primaryWidth\wedge \left\{\subrone=\infty\right\}\right)\cdot\frac{\Pr\left(\subrone<\infty\cond \primaryWidth\right)}{\Pr\left(\subrone=\infty\cond \primaryWidth\right)}\nonumber\\
 &\stackrel{\eqref{probStopWidth}}{\le} \Pr\left(\subrone=\infty\right) \cdot o(1) \nonumber\\
  &\stackrel{\eqref{survival}}{=}o(\eps)\, .
  \label{probStopWidth2}
 \end{align}
 Putting everything together we obtain
 \begin{align}
	\Pr\left(\primaryStop\right)&\stackrel{\eqref{impl1},\,\eqref{impl2},\,\eqref{impl3}}{\leq}
	\Pr\left(\subrone\geq\lcompsize\right)+\Pr\left(\left\{\subrone<\infty\right\}\wedge\primaryWidth\right)\nonumber\\
	&\stackrel{\eqref{probUbrLarge},\, \eqref{probStopWidth2}}{\le}\frac{2\eps}{\cconst}+o\left(\eps\right)\, .
	\label{probPrimaryStop}
 \end{align}
Substituting~\eqref{probPrimaryStop} into~\eqref{secondMomentafterPrimaryStop}, we obtain
\begin{equation}\label{secondMomentafterPrimaryStop2}
\EE (X^2) \le (1+o(1))\frac{2\eps}{\cconst}\binom{n}{j}\sum_{J_2\in \binom{V}{j}\setminus \compone} \Pr\left(J_2 \in \lcomps \cond \primaryStop\right).
\end{equation}

Let us now consider the component containing a given $J_2\in\binom{V}{j}\setminus\compone$. If $\comptwo$ has size at least $\lcompsize$, then certainly the component containing $J_2$ in $\cH^k(n,p)$ has size at least $\lcompsize$. On the other hand, if $\comptwo$ stops because of \eqref{secondaryStopCond1}, then it may be that in fact the whole component is large, but we missed some of it because of the $j$-sets of $\compone$ which we deleted. We would therefore like to say that the number of possible queries between these two sets is small, and so we are unlikely to have missed any edges because of deleting $\compone$.

We first observe that such queries can only occur between $\comptwo$ and the boundary $\boundary$ of $\compone$ (any $j$-sets of $\compone$ not in $\boundary$ were already fully explored). The intuition behind the argument is that $\comptwo$ remains small, while the boundary of $\compone$ is very small, so the number of pairs of $j$-sets, one from each side, should still be small. We might therefore expect that there are very few $k$-sets containing such pairs.

The problem with this basic argument is that the number of $k$-sets containing a pair of $j$-sets is heavily dependent on the size of their intersection. While on the whole most pairs of $j$-sets do not intersect, those which do carry disproportionately large weight because there are many more $k$-sets containing both of them.

We therefore aim to show that $\boundary$ is ``smooth'' in the sense that for any $0\le \ell \le j-1$ and for any $\ell$-set $L$, the number of $j$-sets in $\boundary$ which contain $L$ is about the ``right'' number, and in particular almost the same regardless of the choice of $L$ (though dependent on $|L|=\ell$). This statement is formalised in the following lemma.

Let $\stoptime$ denote the generation at which $\compone$ hits one of the stopping conditions (\eqref{primaryStopCond1}, \eqref{primaryStopCond2}, \eqref{primaryStopCond3}). Recall that $d_L(\generation{i})$ is the degree of $L$ in the $j$-uniform hypergraph with vertex set $V=[n]$ and edge set $\generation{i}$, i.e.\ the number of $j$-sets of the boundary containing $L$.

\begin{lemma}\label{lem:smoothstop}
Conditioned on $\primaryStop$, with probability at least $1-\expprob$, for every $\ell,L$ such that
\begin{itemize}
\item $0 \le \ell \le j-1$;
\item $L$ is an $\ell$-set of vertices;
\end{itemize}
the following holds:
$$d_L(\generation{\stoptime}) = (1\pm o(1))\frac{|\generation{\stoptime}|}{\binom{n}{j}}\binom{n}{j-\ell}.$$ 
\end{lemma}

We note that this lemma is considerably stronger than we would need for the proof of our main result (for which concentration within a constant multiplicative factor would be sufficient). However, since the result is also interesting in itself, we in fact prove a result which is even stronger than Lemma~\ref{lem:smoothstop}.

Consider the process $\ubfs$ starting from an arbitrary $j$-set $J$. Let $\starttime{\ell}$  be the least $i$ such that $|\generationarb{i}| \ge n^{\ell + \delta}$. For each $\ell = 1,\ldots,j-1$, let $\nbrhdbrconst{\ell}:=\frac{\binom{k-\ell}{j-\ell}-1}{\binom{k}{j}-1}$, let $\smoothtime{\ell} := \left\lceil \frac{(j-\ell) \log n}{-\log ((1+2\co + 2\eps)\nbrhdbrconst{\ell})}\right\rceil$, and let $\smoothend{\ell}:= \starttime{\ell}+ \smoothtime{\ell}$.

\begin{lemma}[Smooth boundary lemma]\label{lem:uniformboundary} Let $\eps,p$ be as in Theorem~\ref{thm:main}~\eqref{thm:main:supercrit}.
With probability at least $1-\expprobtwo$, using $\ubfs(J)$
, for every $J,\ell,L,i$ such that
\begin{itemize}
\item $J$ is a $j$-set of vertices;
\item $0 \le \ell \le j-1$;
\item $L$ is an $\ell$-set of vertices;
\item $\smoothend{\ell} \le i \le \stoptime$
\end{itemize}
the following holds:
$$
d_L(\generationarb{i}) = (1\pm \smootherror{\ell})\frac{|\generationarb{i}|}{\binom{n}{j}}\binom{n}{j-\ell}.
$$
\end{lemma}

The intuition behind the various generations which we have defined is as follows: In order for various concentration results to be valid, we will need the average degree of $\ell$-sets to be reasonably large, so $\starttime{\ell}$ is a ``starting time''. Prior to this time we have no information about the degrees of $\ell$-sets, so they may be very non-smooth. Over time, though, any disparity will tend to even itself out, and $\smoothtime{\ell}$ is the time it takes for this process to be complete. From that time $\smoothend{\ell}$ on, the degrees of $\ell$-sets should be smooth.

\begin{remark}
Note that at time $\starttime{\ell}$ we have $\left|\generationarb{\starttime{\ell}}\right| \ge n^{\ell+\delta} \ge n^{1-\delta}$, for $\ell \ge 1$. This means that, by Corollary~\ref{cor:expansionbound} the size of the generations will never decrease. This will be important for various concentration results.
\end{remark}

Lemma~\ref{lem:smoothstop} is almost an immediate corollary of Lemma~\ref{lem:uniformboundary} --  the only issue is that we need to know that $\stoptime \ge \smoothend{\ell}$ for every $\ell=1,\ldots,j-1$, i.e. that we do not reach the stopping conditions before the boundary is smooth. This turns out to be a surprisingly tricky fact to prove, but it will be shown in Section~\ref{sec:smoothstop}.

We will prove Lemma~\ref{lem:uniformboundary} in Section~\ref{sec:uniformboundary}. First, we show how to use Lemma~\ref{lem:smoothstop} to complete the proof of Theorem~\ref{thm:main}.

Conditioned on $\primaryStop$ and $\left\{\subrtwo<\infty\right\},$ we need to analyse the event $\falseNeg$ of $J_2$ being a (potential) \emph{false negative}, i.e.\ $\comptwo$ contains fewer than $\lcompsize$ $j$-sets but the component of $J_2$ in $\cH^k(n,p)$ is larger than $\comptwo$. (In fact, a genuine false negative would require the component to be larger than $\lcompsize$, but bounding the probability of this weaker event will be sufficient.)

We know, by Remark~\ref{rem:overshoot}, that the boundary of $\compone$ has size at most $2\lcompconst \lcompsize$ and each $\ell$-set is contained in $O(\lcompconst \lcompsize /n^\ell)$ sets of the boundary, by Lemma~\ref{lem:smoothstop}. Thus for a $j$-set of $\comptwo$, there are $$\sum_{\ell=0}^{j-1}O(\lcompconst \lcompsize/n^{\ell}) O(n^{k-2j+\ell}) = O(\lcompconst \lcompsize n^{k-2j})$$ $k$-sets which we did not query because they contained $j$-sets of $\compone$. Therefore, if we assume $\comptwo$ contains precisely $r\in\mathbb{N}$ $j$-sets, the expected number of edges within these disallowed $k$-sets is
$$
O\left(\lcompconst \lcompsize n^{k-2j}rp\right)=O\left(\lcompconst^2r\right) 
$$
and thus the probability that we have overlooked at least one edge is $O\left(\lcompconst^2r\right)$. Hence, by the law of total probability we obtain 
\begin{align}
\Pr\left(\falseNeg\cond \primaryStop\wedge\left\{\subrtwo<\infty\right\}\right)&=O\left(\lcompconst^2\sum_{r=0}^{\infty}r\Pr\left(\scomptwo=r\cond \primaryStop\wedge \left\{\subrtwo<\infty\right\}\right)\right)\nonumber\\
&=O\left(\lcompconst^2\EE\left(\scomptwo\cond \primaryStop\wedge \left\{\subrtwo<\infty\right\}\right)\right)\nonumber\\
&=o\left(\eps^2\EE\left(\subrtwo\cond \primaryStop\wedge\left\{\subrtwo<\infty\right\}\right)\right)\nonumber\\
&\stackrel{\ubrtwo\text{ indep.\ of }\primaryStop}{=}o\left(\eps^2\EE\left(\subrtwo\cond \subrtwo<\infty\right)\right) \nonumber\\
&=o\left(\eps^2\EE\left(\sdualbr\right)\right) \nonumber\\
&\stackrel{\eqref{expDual}}{=}o(\eps)
\, .
\label{probFalseNeg}
\end{align}
This solves the main difficulty (apart from the proofs of Lemmas~\ref{lem:uniformboundary} and~\ref{lem:smoothstop}). 

Since we only need an upper bound for the probability that the component of $J_2$ in $\cH^k(n,p)$ is large, we do not need to care about \emph{false positives}. Therefore we consider $J_2$ to be large in the following three cases: 
\begin{itemize}
\item $\ubrtwo$ survives;
\item $\ubrtwo$ dies out and $J_2$ is a false negative;
\item $\ubrtwo$ dies out, $J_2$ is not a false negative and its component in $\cH^k(n,p)$ is large.
\end{itemize} 
Observe that every $j$-set $J_2\in\lcomps$ will satisfy exactly one of these conditions.

 Still assuming that $\primaryStop$ holds, we calculate the probabilities of these events. For the first case we obtain
 \begin{equation}
 \Pr\left(\subrtwo=\infty\cond\primaryStop\right)=\Pr\left(\subrtwo=\infty\right)\stackrel{\eqref{survival1}}{\sim}\frac{2\eps}{\cconst}\, ,
 \label{probCase1}
 \end{equation}
 since the branching process $\ubrtwo$ is independent of $\primaryStop.$ Moreover,  estimate~\eqref{probFalseNeg} immediately shows
 \begin{equation}
  \Pr\left(\left\{\subrtwo<\infty\right\}\wedge\falseNeg\cond \primaryStop\right)\leq \Pr\left(\falseNeg\cond \primaryStop\wedge\left\{\subrtwo<\infty\right\}\right)\stackrel{\eqref{probFalseNeg}}{=}o(\eps)\, .
  \label{probCase2}
 \end{equation} 
  For the last case let us first note that
$$\left(\neg \falseNeg \wedge J_2\in\lcomps\right) \implies \left\{\scomptwo\geq \lcompsize\right\}\implies \left\{\subrtwo\geq \lcompsize\right\} $$
and thus we have
\begin{align}
 \Pr\left(\left\{\subrtwo<\infty\right\}\wedge\neg \falseNeg\wedge\left\{J_2\in\lcomps\right\}\cond\primaryStop\right)&\leq \Pr\left(\lcompsize\leq\subrtwo<\infty\cond\primaryStop\right)\nonumber\\
 &\stackrel{\ubrtwo\text{ indep.\ of }\primaryStop}{=}\Pr\left(\lcompsize\leq\subrtwo<\infty\right)\nonumber\\
  &\le\Pr\left(\subrtwo\geq\lcompsize\cond \subrtwo<\infty\right)\nonumber\\
  &=\Pr\left(\sdualbr\geq\lcompsize\right)\nonumber\\
  &\stackrel{\eqref{expDual}}{\le}(1+o(1))\eps^{-1}/\lcompsize=o(\eps)\, ,
 \label{probCase3}
 \end{align}
 by Markov's Inequality and since $\eps^2\lcompsize\to\infty.$ Consequently, by estimates \eqref{probCase1}, \eqref{probCase2} and \eqref{probCase3}, we obtain
\begin{align}
\Pr\left(J_2\in\lcomps\cond \primaryStop\right)\leq \frac{2\eps}{\cconst}+o(\eps)\, .
\label{probLargeafterPrimaryStop}
\end{align}
This yields a good enough upper bound on the second moment:
\begin{align*}
\EE\left(X^2\right)&\stackrel{\eqref{secondMomentafterPrimaryStop2}}{\leq} (1+o(1)) \binom{n}{j}\frac{2\eps}{\cconst}\sum_{J_2\in\binom{V}{j}\setminus \compone}\Pr\left(J_2\in\lcomps\cond \primaryStop\right)\\
&\stackrel{\eqref{probLargeafterPrimaryStop}}{\leq} (1+o(1))\left(\frac{2\eps}{\cconst}\binom{n}{j}\right)^2\\
&\stackrel{\eqref{firstMoment}}{=}(1+o(1))\EE\left(X\right)^2\, ,
\end{align*}
and therefore we have for the variance
$$\Var\left(X\right)= \EE(X^2)-\EE (X)^2 = o(\EE (X)^2)\,.  $$
Hence Chebyshev's Inequality tells us that for any constant $\zeta>0$, 
$$\Pr (X \neq (1\pm \zeta)\EE(X)) \le \frac{2\, \Var(X)}{\zeta^2\, \EE(X)^2} = o(1/\zeta^2)=o(1),$$ and so whp 
$$X=(1\pm o(1))\EE(X)=(1\pm o(1))\frac{2\eps}{\binom{k}{j}-1}\binom{n}{j}.$$

We have thus shown that the number of $j$-sets in large components is approximately as expected

\subsection{Sprinkling}

However, we also need to know that they all (or at least almost all) $j$-sets of $\lcomps$ lie in the same component. To prove this, we use a standard sprinkling argument.

We first borrow a lemma from~\cite{CoKaPe14}, which states that we may pick out a subsequence of queries  and treat it like an interval of the search process. To this end we denote by $X_t$, for each $t=1,\dots,\binom{n}{k}$, the indicator random variable associated to the $t$-th query of the search process. We will be considering a random subsequence $t_1,t_2,\ldots,t_s$ from $[\binom{n}{k}]$. We say ``$t_i$ is determined by the values of $X_1,\ldots,X_{t_i-1}$'' to mean the following: For any $j$, whether the event $\{t_i=j\}$ holds is determined by the values of $X_1,\ldots,X_{j-1}$. In particular this means that $t_i$ is chosen before $X_{t_i}$ is revealed.

\begin{lemma}\label{lem:subsequence}
Let $S =(t_1,t_2,\ldots,t_s)$ be a (random, ordered) index set chosen according to some criterion such that
\begin{itemize}
\item $t_i$ is determined by the values of $X_1,\ldots,X_{t_i-1}$;
\item with probability $1$ we have $1\le t_1 < t_2 < \ldots < t_s \le \binom{n}{k}$.
\end{itemize}
Then $(X_{t_1},\ldots,X_{t_s})$ is distributed as $(Y_1,\ldots,Y_s)$, where $Y_1,\ldots,Y_s$ are independent Be$(p)$ variables. In particular, we may apply a Chernoff bound to $\sum_{i \in S}X_i$.
\end{lemma}

This lemma applies in particular to subsequences given by the set of queries to $k$-sets containing a particular $\ell$-set while exploring any component, and we therefore have the following corollary.

\begin{corollary}\label{cor:compqueries}
With probability at least $1-\expprob$ for every $1\le \ell \le j-1$, for every $\ell$-set $L$ and for every component $\mathcal{C}$, if the number of queries from $j$-sets in $\mathcal{C}$ to a $k$-set containing $L$ is $x \ge n^{k-j+\delta}$, then  the number of edges in $\mathcal{C}$ containing $L$ is $(1\pm o(1))px$.
\end{corollary}

\begin{proof}
First fix $\ell$, $L$ and $\mathcal{C}$. By Lemma~\ref{lem:subsequence} we may apply a Chernoff bound to the set of queries to $k$-sets containing $L$ within $\mathcal{C}$. Then the probability that the number of edges in $\mathcal{C}$ containing $L$ is not in $(1\pm \zeta)px$, for some constant $\zeta >0$, is at most
\begin{align*}
2\exp \left(-\frac{(px\zeta)^2}{3px} \right) & = 2\exp \left( - px\zeta^2/3 \right)\\
& \ge 2\exp \left( -n^{\delta}\zeta^2/3 \right).
\end{align*}
We may now apply a union bound over all choices of $\ell$ and $L$ (of which there are at most $n^j$) and all choices of $\mathcal{C}$ (of which there are at most $n^j$) and deduce that the probability that any one of these choices goes wrong is at most
\[
2n^{2j}\exp \left( -\zeta^2 n^{\delta}/3\delta^2 \right) \le \exp \left( -n^{\delta/2} \right)
\]
as required.
\end{proof}

\begin{claim}\label{claim:fullshadow}
For each $0\le \ell \le j-1$, whp any set $L$ of $\ell$ vertices is contained in at least $\Theta (\lcompconst n^{j-\ell})$ $j$-sets of any component of size at least $\lcompsize$.
\end{claim}

\begin{proof}
We assume that the high probability event from Corollary~\ref{cor:compqueries} holds.
Consider a component $\mathcal{C}$ with $\beta n^j$ $j$-sets, where $\beta \ge \lcompconst$. We prove by induction on $\ell$ that any $\ell$-set is contained in at least $\left(2\binom{k}{j}^2\right)^{-\ell}\beta n^{j-\ell}$ $j$-sets of $\mathcal{C}$. The case $\ell=0$ simply reasserts the fact that $\mathcal{C}$ has size $\beta n^j$, so assume $\ell\ge 1$.

Now let $L'\subseteq L$ be a set of $\ell-1$ vertices, which by the inductive hypothesis lies in $\left(2\binom{k}{j}^2\right)^{1-\ell}\beta n^{j-\ell+1}$ $j$-sets of $\mathcal{C}$. Then for each such $j$-set $J$ there are at least $\binom{n-j-1}{k-j-1}$ $k$-sets containing $J\cup L$ (more if $L\subseteq J$), and each of these $k$-sets will be counted in this way at most $\binom{k}{j}$ times. Thus the number of queries from $j$-sets in $\mathcal{C}$ to a $k$-set containing $L$ is at least
\begin{align*}
\frac{\binom{n-j-1}{k-j-1}\beta n^{j-\ell +1}}{\left(2\binom{k}{j}^2\right)^{\ell-1}\binom{k}{j}} & \ge (1-o(1))\frac{\beta n^{k-\ell}}{2^{\ell-1}\binom{k}{j}^{2\ell-1}(k-j-1)!}.
\end{align*}
Since $\beta n^{k-\ell} \ge \lcompconst n^{k-j+1} \ge n^{k-j+\delta}$, we may apply Corollary~\ref{cor:compqueries} and thus the number of edges we discover is at least
\begin{align*}
(1-o(1))p \frac{\beta n^{k-\ell}}{2^{\ell-1}\binom{k}{j}^{2\ell-1}(k-j-1)!} & \ge \frac{(k-j)!n^{j-k}}{\binom{k}{j}-1} \cdot \frac{\beta n^{k-\ell}}{2^{\ell}\binom{k}{j}^{2\ell-1}(k-j-1)!}\\
& \ge \frac{\beta n^{j-\ell}}{2^{\ell}\binom{k}{j}^{2\ell}}.
\end{align*}
Since $L$ and $\mathcal{C}$ were arbitrary, this proves the inductive step.
\end{proof}

Note that a substantially stronger version of Claim~\ref{claim:fullshadow} follows (with a little care) from the smooth boundary lemma and our arguments in the remainder of this paper (particularly Lemma~\ref{lem:smallstart}). We give this proof here because it is much more elementary and does not rely on the heavy machinery of the smooth boundary lemma.

\begin{claim}\label{claim:uniquegiant}
$Whp$, $\cH^k(n,p)$ has a giant component of size $(1\pm o(1))\frac{2\eps}{\binom{k}{j}-1}\binom{n}{j}$.\linebreak[4] In particular all other components have size $o(\eps n^j)$.
\end{claim}

\begin{proof}
Let $p_2:= \tfrac{(\log n)^2}{\lcompconst^2 n^{k-j+1}}$ and $p_1$ be such that $p_1+p_2-p_1p_2=(1+\eps)p_0=p$. Note that $p_1 = (1+\eps')p_0$, where $\eps' = (1+o(1))\eps$, so $\cH^k(n,p_1)$ also has $(1\pm o(1))\frac{2\eps}{\binom{k}{j}-1}\binom{n}{j}$ $j$-sets in large components.

By Claim~\ref{claim:fullshadow}, each $(j-1)$-set lies in $\Omega (\lcompconst n)$ $j$-sets of any large component. (In particular, this means that the number of large components is already bounded by $O (1/\lcompconst)$.) Suppose we have two distinct large components. Pick a $(j-1)$-set $J'$ and consider the $(k-j+1)$-uniform link hypergraph of $J'$, i.e. the hypergraph on $[n]\setminus J'$ whose edges are all $(k-j+1)$-sets which, together with $J'$, form an edge of $\cH^k(n,p)$. This has two distinct vertex components, each with $\Omega (\lcompconst n)$ vertices. There are therefore $\Omega (\lcompconst^2 n^{k-j+1})$ possible $(k-j+1)$-sets which meet both.

By the choice of $p_1$ we have $\cH^k(n,p_1)\cup \cH^k(n,p_2) = \cH^k(n,p)$. By adding in another round of exposure with probability $p_2$, the probability that these two components of $\cH^k(n,p_1)$ do not merge is at most
\[
(1-p_2)^{\Omega (\lcompconst^2 n^{k-j+1})}\le \exp \left(-\tfrac{(\log n)^2}{\lcompconst^2 n^{k-j+1}}\Omega (\lcompconst^2 n^{k-j+1})\right) \le \exp \left(-(\log n)^{3/2}\right).
\]
Now suppose the components of $H^k(n,p_1)$ are $\mathcal{C}_1,\ldots,\mathcal{C}_s$, where $s = O(1/\lcompconst)$. Then the probability that at least one of these components does not merge with, say, $\mathcal{C}_1$ is at most
\[
s \exp \left( -(\log n)^{3/2} \right) \le \exp \left( -(\log n)^{3/2} + \log(n^{1/3}) \right) = o(1).
\]
Thus there is a single component of size $(1\pm o(1))\frac{2\eps}{\binom{k}{j}-1}\binom{n}{j}$, while this is also the number of $j$-sets in large components, i.e. sprinkling the edges with probability $p_2$ may have created more large components, but they can only have total size $o(\eps n^j)$.
\end{proof}

\section{Smooth Boundary: Proof of Lemma~\ref{lem:uniformboundary}}\label{sec:uniformboundary}

We will prove Lemma~\ref{lem:uniformboundary} by induction on $\ell$. The case $\ell =0$ is trivially true, since the empty set is contained in all edges of the boundary. (Indeed, this would hold even with $\smootherror{0}=0$, but choosing $\smootherror{0} \gg \eo \log n$ as we did will be convenient later on.) From now on we therefore assume that $\ell \ge 1$ and that the statement holds for $0,\ldots,\ell -1$.

The statement of Lemma~\ref{lem:uniformboundary} says that with probability $1-\expprob$, a certain property must hold for \emph{every} initial $j$-set $J$. We note that it is enough to show this for a single initial $j$-set -- then the full generality is implied by a union bound, since $\binom{n}{j}\expprob = \expprob$. Therefore from now on we fix an initial $j$-set $J_0$, and for simplicity we denote $\generationsmooth{i}$ by $\gen{i}$ for any $i$.

Recall that for $1\le \ell \le j-1$ we defined $\starttime{\ell}$ to be the first generation $i$ for which $\gensize{i}\ge n^{\ell+\delta}$, which is when we aim to start the process of smoothing the degrees of $\ell$-sets, while $\smoothtime{\ell}= \left\lceil \frac{(j-\ell) \log n}{-\log ((1+2\co + 2\eps)\nbrhdbrconst{\ell})}\right\rceil$ is the time taken for this smoothing process. Then $\smoothend{\ell}= \starttime{\ell}+\smoothtime{\ell}$.

We will first show that $\smoothend{\ell}<\starttime{\ell+1}$ for each $1\le \ell \le j-2$, i.e. we finish the smoothing process for the $\ell$-sets before we start the smoothing process for the $(\ell+1)$-sets. (This is important because of the inductive nature of the proof.)

So observe that by Corollary~\ref{cor:expansionbound} and Remark~\ref{rem:overshoot} we have
\begin{align*}
\gensize{\smoothend{\ell}} & \le \gensize{\starttime{\ell}}\left( (1+\eps)(1+2\elbr) \right)^\smoothtime{\ell}\\
& \le 2n^{\ell+\delta} \exp \left( 2\eps \smoothtime{\ell} \right)\\
& \le 2n^{\ell+\delta}\exp (\Theta (\eps \log n))\\
& \le n^{\ell + 2\delta}<n^{\ell+1},
\end{align*}
and therefore we have not yet reached generation $\starttime{\ell+1}$.

We now prove the inductive step in Lemma~\ref{lem:uniformboundary} with the help of a number of preliminary claims. Our strategy is as follows. We fix $i$ and examine how the $j$-sets of $\gen{i}$ can influence the degree of $L$ in $\gen{i+1}$. To simplify notation, we write $\deggen{i}{L}$ instead of $d_L(\gen{i})$.

\begin{itemize}
\item A \emph{jump} to $L$ occurs when we query a $k$-set containing $L$ from a $j$-set which did not contain $L$ and the $k$-set forms an edge of $\cH^k(n,p)$. Such an edge contributes at most $\binom{k-\ell}{j-\ell}$ to $\deggen{i+1}{L}$.
\item A \emph{neighbourhood branching} at $L$ occurs when we query any $k$-set from a $j$-set containing $L$ and it forms an edge. Such an edge contributes at most $\binom{k-\ell}{j-\ell}-1$ to $\deggen{i+1}{L}$. 
\end{itemize}
We aim to show the following:
\begin{itemize}
\item The contribution to $\deggen{i+1}{L}$ made by jumps to $L$ is approximately the same for each $L$ (Claim~\ref{lem:jumps}).
\item The contribution to $\deggen{i+1}{L}$ made by neighbourhood branchings at $L$ tends to be smaller than $\deggen{i}{L}$ (Claim~\ref{lem:nbrhdbranchings}).
\item After sufficiently many steps, the degree in the boundary is approximately smooth.
\end{itemize}

The first two steps are essentially a concentration of probability argument, which can only hold with high probability once the boundary is large. This is the reason why we introduced a starting time $\starttime{\ell}$. Furthermore, we need to know that this smoothing process finishes before the stopping conditions of the search process are reached, i.e. that $\stoptime \ge \smoothend{j-1}$. We therefore need an additional technical lemma:

\begin{itemize}
\item The component grown up to the start of the smoothing process is small (Lemma~\ref{lem:smallstart}).
\end{itemize}

At this point we fix $1 \le \ell \le j-1$ and an $\ell$-set $L$. We will show that the degree of $L$ is well-concentrated with a sufficiently high probability that, at the end of the argument, we can apply a union bound over all choices of $L$. We first show that the contribution of jumps is likely to be well-concentrated around its mean. 

\begin{claim}[Smooth jumps]\label{lem:jumps}
Suppose $|\gen{i}| \ge n^{\ell + \delta}$. Then with probability at least $1-\expprob$, the contribution to $\deggen{i+1}{L}$ made by jumps to $L$ is
\[
(1\pm 2\smootherror{\ell-1})(1+\eps) \frac{\binom{k-\ell}{j-\ell}\left(\binom{k}{\ell}-\binom{j}{\ell}\right)}{\cconst}\frac{|\gen{i}|}{\binom{n}{\ell}}  - O\left(\frac{\deggen{i}{L}}{n}\right). 
\]
\end{claim}

Before we prove this claim, we provide an argument for why the main term should intuitively be the correct contribution from jumps to the degree of $L$. To see this, we consider the number of jumps we expect to \emph{all} $\ell$-sets in generation $i+1$. We have $\gensize{i}$ $j$-sets in generation $i$, from each of which we may query approximately $\binom{n}{k-j}$ $k$-sets, and each forms an edge with probability $p$. Thus we expect to find about
\[
p\binom{n}{k-j}\gensize{i} = \frac{1+\eps}{\cconst}\gensize{i}
\]
edges. On the whole, an edge results in a jump to $\binom{k}{\ell}-\binom{j}{\ell}$ $\ell$-sets (any which are contained in the edge but not in the $j$-set from which we made the query). Since there are $\binom{n}{\ell}$ $\ell$-sets in total, the average number of jumps to an $\ell$-set ``should'' be about
\[
\frac{1+\eps}{\cconst}\gensize{i} \frac{\binom{k}{\ell}-\binom{j}{\ell}}{\binom{n}{\ell}}.
\]
Finally, for most jumps to $L$, the number of $j$-sets containing $L$ which become active is $\binom{k-\ell}{j-\ell}$.

\begin{proof}
We consider how many ways we might jump to $L$. Given any $L'\subsetneq L$, there are $\deggen{i}{L'}-\deggen{i}{L}$ $j$-sets containing $L'$ in $\gen{i}$ from which we might jump to $L$. However, this may include some $j$-sets which actually have a larger intersection with $L$ than $L'$. The number of $j$-sets in $\gen{i}$ which have intersection exactly $L'$ with $L$ is also at least
\[
\deggen{i}{L'} - \sum_{L'\subsetneq L'' \subseteq L}\deggen{i}{L''}.
\]
Note that by the inductive hypothesis for Lemma~\ref{lem:uniformboundary}, $\deggen{i}{L'}$ is of order $|\gen{i}|/n^{\ell'}$ (where $\ell'=|L'|$), and similarly for each $L'\subsetneq L'' \subsetneq L$. Furthermore $\deggen{i}{L}\le \deggen{i}{L''}$ by definition. Therefore the number of $j$-sets in $\gen{i}$ which intersect $L$ in exactly $L'$ is
\[
\jumpsdeg{i}{L'}{L} =
\begin{cases}
\deggen{i}{L'} - \deggen{i}{L} & \mbox{if }\ell'=\ell-1 \\
(1-O(1/n))\deggen{i}{L'} & \mbox{otherwise.}
\end{cases}
\]

From each $j$-set $J$ intersecting $L$ in $L'$, we need a lower bound on the number $k$-sets containing $L$ which we may query, and which contain $\cconst$ neutral $j$-sets. So consider the number of such $k$-sets which are not permissible because they contain some other, non-neutral $j$-set $J'$. We make a case-distinction based on the intersection $I:= J'\cap (J\cup L)$. Let $i:=|I|$. Then the number of non-neutral $j$-sets with this intersection is at most $\Delta_{i}(\disc) = O(\lcompconst n^{j-i})$ by Lemma~\ref{lem:approximations}, Property~\eqref{prop:crudemaxdeg}.

The number of $k$-sets containing $L,J$ and $J'$ is of order $n^{k-j-\ell+\ell' - (j-i)}$. Thus the number of non-permissible $k$-sets is at most $O( \lcompconst n^{k-j-\ell+\ell'})$. Note that this is independent of $i$. Furthermore, the number of possible choices for $I$ is at most $\sum_{i=0}^{j-1} \binom{j+\ell-\ell'}{i} = O(1)$.

On the other hand, the total number of possible $k$-sets is $\binom{n-j-\ell+\ell'}{k-j-\ell+\ell'}$, and so the number of queries we may make to $k$-sets containing $L$ and $\cconst$ neutral $j$-sets is $(1-O(\lcompconst))\binom{n}{k-j-\ell+\ell'}$. The probability that each of these queries results in an edge is $p$. Thus the number of edges resulting in jumps from $L'$ to $L$ has distribution $\Bi (N,p)$, where
\[
N=(1-O(\lcompconst))\binom{n}{k-j-\ell+\ell'}\jumpsdeg{i}{L'}{L}
\]
Note that we use $N$ to denote both a lower bound on the number of queries which would result in jumps to $L$ and would contribute $\binom{k-\ell}{j-\ell}$  to the degree of $L$, and also an upper bound on the total number of queries which would result in jumps to $L$. Of course, technically these are different, but both satisfy the asymptotic formula above.

For simplicity, we will assume that
\[
N\ge \Theta(n^{k-j+\delta})
\]
in all cases. This certainly follows from the fact that $\jumpsdeg{i}{L'}{L}\ge \Theta(n^{\ell-\ell'+\delta})$ if $\ell'<\ell-1$. For $\ell'=\ell-1$, if $N=o(n^{k-j+\delta})$ then $\jumpsdeg{i}{L'}{L} = o(n^{1+\delta})$, and it must be that $\deggen{i}{L}= (1- o(1)) \deggen{i}{L'} = \Theta (\gensize{i}/n^{\ell-1})$, and therefore the $O(\deggen{i}{L}/n)$ error term in the statement of Claim~\ref{lem:jumps} is as large as the main term. The lower bound is therefore automatic and we only need to prove an upper bound, which will only become harder if we increase $N$.

By the Chernoff bound (Theorem~\ref{thm:chernoff}), the probability that the number of such edges is not in $(1\pm \eo)pN$ is at most 
\[
2\exp(-\eo^2 Np/3)  \le \exp(-\Theta(\eo^2 n^{\delta})) \le \expprob.
\]

If this unlikely event does not occur, then the contribution to $\deggen{i}{L}$ made by jumps from $L'$ to $L$ is
\begin{align*}
 \binom{k-\ell}{j-\ell}(1\pm \eo)N\frac{1+\eps}{\cconst\binom{n}{k-j}}
=  (1\pm 2\eo)(1+\eps)\frac{\binom{k-\ell}{j-\ell}}{\cconst}\frac{(k-j)!}{(k-j-\ell+\ell')!}\frac{\jumpsdeg{i}{L'}{L}}{n^{\ell-\ell'}}.
\end{align*}
We now sum over all such sets $L' \subsetneq L$ and use the fact that
\begin{align*}
\frac{\jumpsdeg{i}{L'}{L}}{n^{\ell-\ell'}} & = (1- O(1/n))\frac{\deggen{i}{L'}}{n^{\ell-\ell'}}-O\left(\frac{\deggen{i}{L}}{n}\right)
\end{align*}
in all cases.
We also use the induction hypothesis from Lemma~\ref{lem:uniformboundary}, which tells us that
\[
\deggen{i}{L'}= (1\pm \smootherror{\ell'})\frac{\gensize{i}}{\binom{n}{j}}\binom{n}{j-\ell'},
\]
and deduce that the contribution to $\deggen{i+1}{L}$ made by jumps to $L$ is 
\begin{align*}
 \sum_{\ell'=0}^{\ell-1}\binom{\ell}{\ell'}& (1\pm 2\eo)(1+\eps)\frac{\binom{k-\ell}{j-\ell}}{\cconst}\frac{(k-j)!}{(k-j-\ell+\ell')!}
\frac{\jumpsdeg{i}{L'}{L}}{n^{\ell-\ell'}}\\
= \, & (1\pm 3\eo)(1+\eps)\frac{\binom{k-\ell}{j-\ell}}{\cconst} \sum_{\ell'=0}^{\ell-1}\binom{\ell}{\ell'} \frac{(k-j)!}{(k-j-\ell+\ell')!}
\frac{\deggen{i}{L'}}{n^{\ell-\ell'}}-O\left(\frac{\deggen{i}{L}}{n}\right)\\
= \, & (1\pm \frac{3}{2}\smootherror{\ell-1})(1+\eps)\frac{\binom{k-\ell}{j-\ell}}{\cconst} \sum_{\ell'=0}^{\ell-1} \frac{\binom{\ell}{\ell'}(k-j)!j!}{(k-j-\ell+\ell')!(j-\ell')!} \frac{\gensize{i}}{n^\ell}-O\left(\frac{\deggen{i}{L}}{n}\right)\\
= \, & (1\pm 2\smootherror{\ell-1})(1+\eps) \frac{\binom{k-\ell}{j-\ell}}{\cconst}\frac{|\gen{i}|}{\binom{n}{\ell}}\, f -O\left(\frac{\deggen{i}{L}}{n}\right),
\end{align*}
where 
\begin{align*}
 f=f(j,k,\ell):=&\frac{(k-j)!j!}{\ell!} \sum_{\ell'=0}^{\ell-1}\frac{\binom{\ell}{\ell'}}{(j-\ell')!(k-j-\ell+\ell')!}\\ 
 =& \frac{(k-j)!j!}{\ell!} \sum_{\ell'=0}^{\ell-1}\binom{\ell}{\ell'}\frac{\binom{k-\ell}{j-\ell'}}{(k-\ell)!}\\
 =& \frac{(k-j)!j!}{(k-\ell)!\ell!}\left(\sum_{\ell'=0}^{\ell}\binom{\ell}{\ell'}\binom{k-\ell}{j-\ell'}-\binom{k-\ell}{j-\ell}\right)\\
 =& \frac{(k-j)!j!}{(k-\ell)!\ell!}\left(\binom{k}{j} - \binom{k-\ell}{(j-\ell)}\right)\\
 =& \frac{k!}{\ell!(k-\ell)!} - \frac{j!}{\ell!(j-\ell)!}\\
 =& \binom{k}{\ell}-\binom{j}{\ell}
\end{align*}
as required.
\end{proof}

Note that the effect that $\deggen{i}{L}$ has on the number of jumps to $L$ in generation ${i+1}$ is very small (the $O(\deggen{i}{L}/n)$ term). Next we show that the effect that $\deggen{i}{L}$ has on the number of neighbourhood branchings at $L$ in generation $i+1$, while significantly larger, is still likely to be smaller than $\deggen{i}{L}$. Let $\cconstl{\ell}:=\binom{k-\ell}{j-\ell}-1$ for $\ell=1,\ldots,j-1$ (and observe that this is consistent with the previous definition of $\cconst$). Recall that $\nbrhdbrconst{\ell}= \frac{\binom{k-\ell}{j-\ell}-1}{\binom{k}{j}-1} = \frac{\cconstl{\ell}}{\cconst} < 1 $ for $\ell \ge 1$.

\begin{claim}[Neighbourhood branchings contract]\label{lem:nbrhdbranchings}
With probability at least\linebreak[4] $1-\expprobtwo$, the contribution to $\deggen{i+1}{L}$ made by neighbourhood branchings at $L$  is
\[
\begin{cases}
(1\pm \co)(1+\eps) \nbrhdbrconst{\ell} \deggen{i}{L} & \mbox{if } \deggen{i}{L}  \ge n^{\delta/3};\\
\le (1+\co)(1+\eps) \nbrhdbrconst{\ell} n^{\delta/3} & \mbox{otherwise.}
\end{cases}
\]
\end{claim}

Note that the ``bad'' probability in this claim is larger than in the other results we have stated. For technical reasons, we will need to apply Claim~\ref{lem:nbrhdbranchings} for smaller boundaries than, for example, Claim~\ref{lem:jumps}.

\begin{proof}
Let $D$ denote the contribution to $\deggen{i+1}{L}$ made by neighbourhood branchings at $L$. We first prove the upper bound. From every $j$-set containing $L$, we make at most $\binom{n}{k-j}$ queries, and each time we discover an edge in this way, it has a contribution of at most $\cconstl{\ell}$ to $\deggen{i+1}{L}$. Thus $D$ is stochastically dominated by a random variable $Y$ with distribution $\cconstl{\ell}\cdot\Bi \big(\binom{n}{k-j}\deggen{i}{L},p \big)$, which has expectation
\begin{align*}
\EE (Y)= \cconstl{\ell}\binom{n}{k-j}\deggen{i}{L}p & =(1+\eps) \frac{\cconstl{\ell}}{\cconst}\deggen{i}{L}\\
& =(1+\eps)\nbrhdbrconst{\ell}\deggen{i}{L}.
\end{align*}
Thus when $\deggen{i}{L}\ge n^{\delta/3}$, by the Chernoff bound (Theorem~\ref{thm:chernoff}) applied to $Y$ we have
\begin{align*}
\Pr (D \ge (1+\co)(1+\eps)\nbrhdbrconst{\ell}\deggen{i}{L}) & \le \Pr (Y \ge (1+\co)(1+\eps)\nbrhdbrconst{\ell}\deggen{i}{L})\\
& \le \exp \left(- \co^2 (1+\eps) \nbrhdbrconst{\ell}\deggen{i}{L}/3 \right)\\
& \le \exp (-\Theta (\co^2 \deggen{i}{L}))\\
& \le \expprobtwo
\end{align*}
since $\co \ge n^{-\delta/24}$.
This proves the upper bound in the case that $\deggen{i}{L} \ge n^{\delta/3}$. In the second case, we simply take $\cconstl{\ell} \cdot \Bi \big(\binom{n}{k-j}n^{\delta/3},p \big)$ as a dominating variable and a similar calculation holds. This proves the upper bound in both cases.

For the lower bound, we observe that since $\lbr \prec \ubfs$, the number of queries that would result in exactly $\cconstl{\ell}$ branchings that we make from each $j$-set is at least $(1-\elbr)\binom{n}{k-j}$. Thus $X$ dominates a random variable $Z$ with distribution $\cconstl{\ell} \cdot \Bi \big((1-\elbr)\binom{n}{k-j}\deggen{i}{L},p \big)$. A similar calculation shows us that 
\begin{align*}
\Pr (D \le (1-\co)(1+\eps)\nbrhdbrconst{\ell}\deggen{i}{L}) & \le \Pr (Z \le (1-\co)(1+\eps)\nbrhdbrconst{\ell}\deggen{i}{L})\\
& \stackrel{\mbox{{\tiny (Thm \ref{thm:chernoff})}}}{\le} \exp \left(- (\co-\elbr)^2 (1+\eps)(1-\elbr) \nbrhdbrconst{\ell}\deggen{i}{L}/2 \right)\\
& \le \exp (-\Theta (\co^2 \deggen{i}{L}))\\
& \le \expprobtwo.\qedhere
\end{align*}
\end{proof}

We now complete the proof of Lemma~\ref{lem:uniformboundary}.  We have already observed that $\starttime{\ell+1}\ge \smoothend{\ell}$ and that the statement of Lemma~\ref{lem:uniformboundary} holds for $\ell=0$, so by the inductive hypothesis we assume that the conclusion of Lemma~\ref{lem:uniformboundary} holds for $\ell'=0,\ldots,\ell-1$ and for all generations $i\in [\starttime{\ell}, \stoptime]$. By Corollary~\ref{cor:expansionbound}, we know that for this range of $i$, $\gen{i}$ is large enough to apply all of our concentration results. 

Pick any $\varstarttime{\ell}\in [\starttime{\ell},\stoptime]$. For $s\in [\smoothtime{\ell} , \stoptime - \varstarttime{\ell}]$, we let $d_s:= \deggen{\varstarttime{\ell}+s}{L}$ and let $d_s':=\max (d_s,n^{\delta/3})$. We assume that all of the high probability events of the previous claims hold (this occurs with probability at least $1-\expprobtwo$). Then we have
\begin{align*}
d_s & \le (1+\co)(1+\eps)\nbrhdbrconst{\ell} d_{s-1}' + (1 + 2\smootherror{\ell-1})(1+\eps) \frac{\binom{k-\ell}{j-\ell}\left(\binom{k}{\ell}-\binom{j}{\ell}\right)}{\cconst} \frac{|\gen{\varstarttime{\ell}+s-1}|}{\binom{n}{\ell}}  \\
& = (1+\co)(1+\eps)\nbrhdbrconst{\ell}d_{s-1}' + (1+ 2\smootherror{\ell-1})g\, |\gen{\varstarttime{\ell}+s-1}|
\end{align*}
where
\[
g=g(k,j,\ell,n,\eps,\smootherror{\ell-1}):= (1+\eps) \frac{\binom{k-\ell}{j-\ell}\left(\binom{k}{\ell}-\binom{j}{\ell}\right)}{\cconst } \binom{n}{\ell}^{-1}.
\]
Note that $g$ is not dependent on $s$, and that $(1+\co)(1+\eps)\nbrhdbrconst{\ell}<1$. If $d_{s-1}'=d_{s-1}$, we apply the same inequality with an index shift, and keep iterating until we have some $d_{s'} \le n^{\delta/3}$ or else we reach $s'=0$. In this way we obtain
\begin{align*}
& d_s \le (1+\co)^s(1+\eps)^s \nbrhdbrconst{\ell}^s  d_0 + n^{\delta/3}\\
 &  \hspace{1.5cm} + (1+ 2\smootherror{\ell-1})g \sum_{s'=0}^s (1+\co)^{s'}(1+\eps)^{s'} \nbrhdbrconst{\ell}^{s'}|\gen{\varstarttime{\ell}+s-s'}| \\
& \le (1+\co)^s\left( (1+\eps)^s \nbrhdbrconst{\ell}^s  d_0 +(1+ 2\smootherror{\ell-1})g \sum_{s'=0}^s (1+\eps)^{s'} \nbrhdbrconst{\ell}^{s'}|\gen{\varstarttime{\ell}+s-s'}|  \right) + n^{\delta/3}.
\end{align*}
To calculate the corresponding lower bound we cannot use $(1-\co)(1+\eps)\nbrhdbrconst{\ell}d_{s-1}$ from Lemma~\ref{lem:nbrhdbranchings}, since it may be that $d_{s-1}<n^{\delta/3}$, in which case that lemma does not give us any lower bound. Instead we use the lower bound (in all cases) of $(1-\co)(1+\eps)\nbrhdbrconst{\ell}d_{s-1}-n^{\delta/3}$, which may be negative but which will turn out to be good enough. Thus we have
\begin{align*}
d_s &\ge (1-\co)(1+\eps)\nbrhdbrconst{\ell} d_{s-1}-n^{\delta/3} \\
	&\qquad\qquad+ \left(1 - 2\smootherror{\ell-1}\right)(1+\eps) \frac{\binom{k-\ell}{j-\ell}\left(\binom{k}{\ell}-\binom{j}{\ell}\right)}{\cconst} \frac{|\gen{\varstarttime{\ell}+s-1}|}{\binom{n}{\ell}} -  O\left(\frac{d_{s-1}}{n}\right) \\
& \ge (1-2\co)(1+\eps)\nbrhdbrconst{\ell}d_{s-1} + (1- 3\smootherror{\ell-1})g\, |\gen{\varstarttime{\ell}+s-1}|.
\end{align*}
Note that we have absorbed the $-n^{\delta/3}$ term by modifying the $\smootherror{\ell-1}$ term. This is permissible since $\gensize{i}\ge n^{\ell+\delta}$, and so $\smootherror{\ell-1} \gensize{i}/n^\ell \gg n^{\delta/3}$.
Iterating gives
\begin{align*}
d_s & \ge (1-2\co)^s(1+\eps)^s \nbrhdbrconst{\ell}^s  d_0\\
& \hspace{1.5cm} + (1- 3\smootherror{\ell-1})g \sum_{s'=0}^s (1-2\co)^{s'}(1+\eps)^{s'} \nbrhdbrconst{\ell}^{s'}|\gen{\varstarttime{\ell}+s-s'}|\\
& \ge (1-2\co)^s \left( (1+\eps)^s \nbrhdbrconst{\ell}^s  d_0 + (1- 3\smootherror{\ell-1})g \sum_{s'=0}^s (1+\eps)^{s'} \nbrhdbrconst{\ell}^{s'}|\gen{\varstarttime{\ell}+s-s'}|  \right).
\end{align*}
Observe that only the first term of both the upper and lower bound on $d_s$ depend on $L$. Additionally, for $s\ge \smoothtime{\ell}$,we have
\begin{align*}
0\le (1-2\co)^s (1+\eps)^s \nbrhdbrconst{\ell}^s d_0 
  &\le(1+\co)^s (1+\eps)^s \nbrhdbrconst{\ell}^s d_0\\
  &\le \left(\left(1+2\co + 2\eps\right)\nbrhdbrconst{\ell}\right)^\smoothtime{\ell} n^{j-\ell}
 \le 1
\end{align*}
since we chose $\smoothtime{\ell} = \left\lceil \frac{(j-\ell) \log n}{-\log \left(\left(1+2\co + 2\eps\right)\nbrhdbrconst{\ell}\right)}\right\rceil$. In other words, $d_0$, the degree of $L$ at time $\varstarttime{\ell}$ only has an influence of at most one by time $\varstarttime{\ell}+s$, which will not affect calculations significantly, so we ignore it.

The remaining upper and lower bounds do not depend on $L$. Furthermore, observing that $g=\Theta (n^{-\ell})$ and that $\gensize{i}\ge \Theta(n^{\ell+\delta})$, we have $n^{\delta/3}=O(n^{-2\delta/3}g\gensize{i})$, and so the remaining upper and lower bounds differ by a multiplicative factor of
\begin{align*}
\frac{(1+O(n^{-2\delta/3}))(1+\eo)^s (1+2\smootherror{\ell-1})}{(1-2\eo)^s (1-3\smootherror{\ell-1})}
 & \le (1+5\co)^{s}(1+6\smootherror{\ell-1})\\
 & \le (1+7\smootherror{\ell-1}).
\end{align*}
By definition we have
\[
1+7\smootherror{\ell-1} \le 1+\smootherror{\ell}.
\]
Taking a union bound over all sets $L$ of size $\ell$ and all possible starting generations $\varstarttime{\ell}$, we may say that with probability at least $1-\expprobtwo$, all $\ell$-sets have asymptotically the same degree in $\gen{i}$. More precisely we have
\[
\sum_L \deggen{i}{L} = \binom{j}{\ell}\gensize{i}
\]
which implies that
\[
\frac{1}{\binom{n}{\ell}}\sum_L \deggen{i}{L} = \frac{\binom{n-\ell}{j-\ell}}{\binom{n}{j}}\gensize{i}.
\]
We further have, for any particular $\ell$-set $L_0$, that by the arguments above
\[
\frac{1}{1+\smootherror{\ell}}\frac1{\binom{n}{\ell}}\sum_L \deggen{i}{L} \le  \deggen{i}{L_0}\le(1+\smootherror{\ell}) \frac{1}{\binom{n}{\ell}}\sum_L \deggen{i}{L} 
\]
and since $1/(1+\smootherror{\ell})\ge 1-\smootherror{\ell}$, the conclusion follows.
The proof of Lemma~\ref{lem:uniformboundary} is now complete except for the proof of Lemma~\ref{lem:smoothstop}.\qed

Note that the introduction of $\varstarttime{\ell}$, which can be any generation after the fixed one $\starttime{\ell}$, ensures that our concentration does not get worse further along the process due to compounding error terms. Instead, we can imagine starting the smoothing process at any time, and choose the one that gives us the best possible concentration. Of course, in reality this means that once things are smooth, they remain smooth.

\section{Smoothness at stopping time: Proof of Lemma~\ref{lem:smoothstop}}\label{sec:smoothstop}

Recall that $\primaryStop$ is the event that stopping conditions~$\eqref{primaryStopCond2}$ or~\eqref{primaryStopCond3} were implemented, i.e. $\compone$ became large or its boundary became large before the component was fully explored.

In order to complete the proof of Lemma~\ref{lem:smoothstop}, we need to know that, conditioned on $\primaryStop$, with high probability $\stoptime \ge \smoothend{j-1}$, i.e. the smoothing process is finished before the stopping conditions are reached. If this were not true, then Lemma~\ref{lem:uniformboundary} would be an empty statement and therefore of no help.

In order to prove that it is not, we will need the following two technical lemmas. The first roughly states that the initial expansion happens fast enough that the portion of the component discovered before the start of the smoothing process (for $(j-1)$-sets) is negligible.

\begin{lemma}\label{lem:smallstart}
With probability at least $1-\expprob$, for every $\ell=0,\ldots,j-1$, we have $$\Delta_\ell \left(\compone\left(\starttime{j-1}\right)\right) = o(\lcompconst n^{j-\ell}).$$
\end{lemma}

The second lemma tells us that there will be enough time between the start of the smoothing process and the stopping conditions being reached for the smoothing process to complete.

\begin{lemma}\label{lem:enoughsmoothingtime}
Conditioned on $\primaryStop$, with probability at least $1-\expprob$, we have $\stoptime - \starttime{j-1} \ge \frac{\delta}{2} \eps^{-1}\log n$
\end{lemma}

We will prove these two lemmas in Sections~\ref{sec:smallstart} and~\ref{sec:enoughsmoothingtime} respectively. We first show how together they imply Lemma~\ref{lem:smoothstop}.

We note that by definition $\smoothend{j-1} - \starttime{j-1} = \smoothtime{j-1} = O(\log n)$, where the constant depends only on $k,j$, and therefore this is substantially smaller than $\eps^{-1}\log n$. Thus conditioned on $\primaryStop$, with probability at least $1-\expprob$ we have $\stoptime \ge \smoothend{j-1}$ and therefore Lemma~\ref{lem:smoothstop} follows directly from Lemma~\ref{lem:uniformboundary}.

Thus the proof of Lemma~\ref{lem:smoothstop}, and therefore of Theorem~\ref{thm:main}, is complete except for the proof of Lemmas~\ref{lem:smallstart} and~\ref{lem:enoughsmoothingtime}.\qed

\subsection{Proof of Lemma~\ref{lem:enoughsmoothingtime}}\label{sec:enoughsmoothingtime}

In this section we will use Lemma~\ref{lem:smallstart} to prove lemma~\ref{lem:enoughsmoothingtime}.

The expansion from one generation to the next is at most $(1+2\elbr)(1+\eps)\le 1+2\eps$ by Corollary~\ref{cor:expansionbound}. Let $x:= \stoptime - \starttime{j-1}$. Suppose first that we hit the stopping condition \eqref{primaryStopCond3}, i.e.\ $\left|\boundary\right|\ge \lcompconst^2 n^j$. Since by Remark~\ref{rem:overshoot} $\gensize{\starttime{j-1}}\le 2n^{j-1+\delta}$, we then have $(1+2\eps)^x \ge \lcompconst^2 n^{1-\delta}/2$, which gives
\[
x \ge \frac{\log(\lcompconst^2 n^{1-\delta}/2)}{\log (1+2\eps)} \ge \frac{\log (n^\delta)}{2\eps} = \frac{\delta}{2} \eps^{-1}\log n.
\]
On the other hand suppose we hit the stopping condition $|\compone|\ge \lcompconst n^j$. Then  Lemma~\ref{lem:smallstart} (for $\ell=0$) implies that up to time $\starttime{j-1}$ we have only seen $o(\lcompconst n^j)$ $j$-sets in total, and thus by Remark~\ref{rem:overshoot} we have
\begin{align*}
&&\sum_{i=0}^x (1+2\eps)^{i} 2 n^{j-1+\delta} &\ge (1-o(1))\lcompconst n^j\\
 &\Rightarrow&\frac{1-(1+2\eps)^{x+1}}{1-(1+2\eps)}  &\ge  (1-o(1))\lcompconst n^{1-\delta}/2\\
&\Rightarrow & (1+2\eps)^{x}&\ge  (1-o(1))\eps \lcompconst n^{1-\delta} \ge n^{\delta}\\
&\Rightarrow & x\log(1+2\eps)  &\ge  \delta \log n \\
&\Rightarrow & x  &\ge  \frac{\delta}{2} \eps^{-1}\log n
\end{align*}
as required. Note that in all the Lemmas, Corollaries and Remarks that we used here, the ``good'' event holds with probability $1- \expprob$, and a union bound over all failure probabilities still leaves a probability of at least $1-\expprob$. \qed

\subsection{Proof of Lemma~\ref{lem:smallstart}}\label{sec:smallstart}

The critical tool in our proof of Lemma~\ref{lem:smallstart} is Lemma~\ref{lem:getbig}, which gives a lower bound on the probability that the process will become large within a certain number of steps. Recall that $\lbr$ is a branching process starting with one individual in the $0$-th generation and whose offspring distribution $\adistr$ is the distribution of the random variable $$\cconst \cdot \Bi \left((1-\elbr)\binom{n}{k-j} ,p\right).$$ We have also seen that this provides a lower coupling on our actual search process during the time period that we are interested in.

Let $\annconst>0$ be some constant and let $\anntime:=(j-1+\delta+\annconst)\eps^{-1}\log n $.  For any $i\in\Nat$ we denote by $\lbrgensize{i}$ the size of the $i$-th generation of $\lbr$.

\begin{lemma}\label{lem:getbig} 
With probability at least $\eps/n^\annconst$, $\lbrgensize{\anntime} \ge n^{j-1+\delta}$.
\end{lemma}

We aim to imitate a proof of Markov's inequality. However, since we have no upper limit on the size of the boundary, we will need to limit the contribution that the upper tail makes to the expected size of the boundary.

In order to do this we need to bound the upper tail probabilities, preferably exponentially. For this we imitate a proof of a Chernoff bound, applying Markov's inequality to the random variable $e^{\annalpha \lbrgensize{\anntime}}$ for some well chosen $\annalpha$.

We therefore need to calculate $\EE \left(e^{\annalpha \lbrgensize{i}}\right).$  Let $\varphi(z):= \sum_{t=0}^\infty e^{zt}\Pr (\adistr=t)$ be the moment generating function of $\adistr$. We also define 
$$\annalpha=\annalpha(n)=\frac{\eps}{(1+\eps)^\anntime n^{\annconst/2}}$$
and $C_i$ recursively by
\begin{align*}
\annerror{0} &:= 1+\delta\\
\annerror{i} & :=1+ \annexcess (i+1)\left(\frac{\cconst}{2}\annerror{i-1} (1+\eps)^i\annalpha+\eps\right) \hspace{0.5cm} \mbox{for } i\ge 1.
\end{align*}

\begin{claim}
For all $i$ we have
\[
\EE \left(e^{\annalpha \lbrgensize{i}}\right) = \varphi\left(\log \left(\EE\left(e^{\annalpha \lbrgensize{i-1}}\right)\right)\right).
\]
In particular,
\begin{equation}\label{IndAsymp}
\EE \left(e^{\annalpha \lbrgensize{i}}\right)\le 1+\annerror{i}\left(1+\eps\right)^{i}\annalpha\, ,
\end{equation}
for sufficiently large $n$ as long as 
\begin{equation}\label{IndStop}
i\left(1+\eps\right)^{i}\annalpha \log n=o(1)\, .
\end{equation}
\end{claim} 

\begin{proof} 
The first part of the claim is a standard fact about moment generating functions, which can be proved by conditioning on the number of children in the first step:
\begin{align*}
\EE \left(e^{\annalpha \lbrgensize{i}}\right) & = \sum_{t=0}^\infty \Pr \left(\lbrgensize{1}=t\right) \EE \left(e^{\annalpha \lbrgensize{i}} \; | \; \lbrgensize{1}=t\right)\\
& = \sum_{t=0}^\infty \Pr \left(\lbrgensize{1}=t\right) \EE \left(e^{\annalpha \lbrgensize{i}} \; | \; \lbrgensize{1}=1\right)^t\\
& = \sum_{t=0}^\infty \Pr \left(\lbrgensize{1}=t\right) \EE \left(e^{\annalpha \lbrgensize{i-1}}\right)^t\\
& = \varphi\left(\log \left(\EE\left(e^{\annalpha \lbrgensize{i-1}}\right)\right)\right).
\end{align*}
In proving the second part of the claim, for simplicity we will actually calculate the expectation in $\ubr$ rather than $\lbr$, which amounts to ignoring the $(1-\elbr)$ terms. This is permissible since we only require an upper bound on $\EE (e^{\annalpha \lbrgensize{i}})$ (and is a good approximation in any case since $\elbr \ll \eps$).

We prove the second part of the claim by induction on $i$. For the case $i=0$ the statement becomes $e^\annalpha \le 1+\annerror{0}\annalpha$ which is certainly true for sufficiently large $n$ since $\annalpha =o(1)$ and $\annerror{0} >1$ is bounded away from $1$. We therefore assume that the result holds for $i-1$.

One fact which we will use repeatedly is that for $i\le \anntime$ we have
\begin{equation}
\annerror{i} = O(\log n).
\end{equation}
This can easily seen by induction on $i$, since
\begin{align*}
\annerror{i} & = 1 + O(\anntime(\annerror{i-1}(1+\eps)^\anntime\annalpha +\eps))\\
& = 1+ O(\anntime(\annerror{i-1}\eps/n^\annconst + \eps))\\
& = 1+O(\anntime\eps)\\
& = O(\log n).
\end{align*}
To simplify notation we define $\annea{i}:= (1+\eps)^i \annalpha$. Also, let
\[
p_i:=\left(1+ \annerror{i-1}\annea{i-1}\right)^\cconst p
\]
and note that for $i\le \anntime$ we have $p\le p_i \le p_\anntime = (1+o(1))p$ (and in particular $p_i\le 1$). We now have (using $t=\cconst s$ above and the induction hypothesis)
\begin{align*}
\EE\left(e^{\annalpha \lbrgensize{i}}\right) & = \sum_s \binom{\binom{n}{k-j}}{s}p^s (1-p)^{\binom{n}{k-j}-s}\, \EE\left(e^{\annalpha \lbrgensize{i-1}}\right)^{\cconst s} \\
& \le \sum_s \binom{\binom{n}{k-j}}{s}p^s (1-p)^{\binom{n}{k-j}-s}\left(1+\annerror{i-1}\annea{i-1}\right)^{\cconst s} \\
& = \sum_s \binom{\binom{n}{k-j}}{s} p_i^s (1-p)^{\binom{n}{k-j}-s} \\
& \le \left(\frac{1-p}{1-p_i}\right)^{\binom{n}{k-j}}\sum_s \binom{\binom{n}{k-j}}{s} p_i^s (1-p_i)^{\binom{n}{k-j}-s}.
\end{align*}
The terms in the sum are simply binomial probability expressions, now with a slightly different probability, and therefore their sum is $1$. Thus we obtain
\begin{align*}
\EE\left(e^{\annalpha \lbrgensize{i}}\right) & \le \left(\frac{1-p}{1-p_i}\right)^{\binom{n}{k-j}}\\
& = \left(1+(p_i-p)(1+p_i + p_i^2 + \ldots)\right)^{\binom{n}{k-j}}\\
& \le 1 + \binom{n}{k-j}p\left(\frac{p_i}{p}-1\right)(1+p_i + \ldots) \\
&\qquad+ \frac{\binom{n}{k-j}^2}{2}p^2 \left(\frac{p_i}{p}-1\right)^2(1+p_i+\ldots)^2 + O\left(\left(\frac{p_i}{p}-1\right)^3\right).
\end{align*}
Now
\[
\frac{p_i}{p}-1 = \cconst \annerror{i-1} \annea{i-1} + \binom{\cconst}{2}\annerror{i-1}^2 \annea{i-1}^2 + O(\annerror{i-1}^3\annea{i-1}^3)
\]
so we have (using the fact that $\annea{i-1}\le \annea{i}$)
\begin{align*}
\EE\left(e^{\annalpha \lbrgensize{i}}\right)  \le \; & 1+ (1+\eps)\left(\annerror{i-1}\annea{i} + \frac{\cconst-1}{2}\annerror{i-1}^2 \annea{i}^2 + O(\annerror{i-1}^3\annea{i}^3) \right)(1+p_i+\ldots)\\
& + (1+\eps)^2\left(\frac{\annerror{i-1}^2}{2}\annea{i}^2 + O(\annerror{i-1}^3 \annea{i}^3)\right)(1+p_i+\ldots)^2 + O(\annerror{i-1}^3\annea{i}^3)\\
\le \; & 1 + \annerror{i-1}\annea{i} \left(1 + \eps + p_i + \frac{\cconst}{2}\annerror{i-1}\annea{i} + O(\annerror{i-1}^2\annea{i}^2 + p_i^2 + \eps^2) \right).
\end{align*}
Finally, let us observe that, provided $i(1+\eps)^i \annalpha =o(1)$, we have
\begin{align*}
& \annerror{i-1} \left(1 + \eps + p_i + \frac{\cconst}{2}\annerror{i-1}\annea{i} + O(\annerror{i-1}^2 \annea{i}^2 + p_i^2 + \eps^2) \right)\\
&\qquad=  \left(1+\annexcess i\left(\frac{\cconst}{2}\annerror{i-2}\annea{i-1} + \eps\right)\right)\left(1+\eps + \frac{\cconst}{2}\annerror{i-1}\annea{i} + o(\eps + \annerror{i-1}\annea{i})\right).
\end{align*}
Using the facts that $\annerror{i-2}\le \annerror{i-1}$ and $\annea{i-1} \le \annea{i}$, this gives us
\begin{align*}
& \annerror{i-1} \left(1 + \eps + p_i + \frac{\cconst}{2}\annerror{i-1}\annea{i} + O(\annerror{i-1}^2 \annea{i}^2 + p_i^2 + \eps^2) \right)\\
&\qquad\le  1 + \annerror{i-1}\annea{i} \left(\annexcess i\frac{\cconst}{2} + \frac{\cconst}{2} + o(1)\right) + \eps\left(\annexcess i + 1 + o(1)\right)\\
&\qquad\le  1 + \annerror{i-1}\annea{i} \annexcess (i+1)\frac{\cconst}{2} + \eps\annexcess (i+1)\\
&\qquad=  \annerror{i}
\end{align*}
as required.
\end{proof}

\begin{proof}[Proof of Lemma~\ref{lem:getbig}]
Now for any number $a$, let
\[
q_a:=\Pr (\lbrgensize{\anntime}\ge a) 
\]
We set $z:=n^{j-1+\delta}$ and $y_i:=(1+\eps)^\anntime n^{i\annconst}/(2\eps)$ for every $i=1,2,\ldots$ and note that
\begin{align*}
y_1 & = n^{(j-1+\delta + \annconst)\eps^{-1}\log (1+\eps)} n^\annconst /\eps \\
& \ge n^{j-1+\delta + 3\annconst/2}/\eps \\
& > z.
\end{align*}
Our main aim is to find a lower bound on $q_z$. We observe that
\begin{equation}\label{eq:markovtype}
\EE (\lbrgensize{\anntime}) \le (1-q_z)  z + q_z  y_1 + \sum_{i=1}^\infty q_{y_i}y_{i+1}.
\end{equation}
Recall that
$$\annerror{\anntime}\anntime\left(1+\eps\right)^\anntime\annalpha =o(1) $$
and thus Markov's inequality tells us that
\[
q_{y_i}\le \frac{\EE \left(e^{\annalpha \lbrgensize{\anntime}}\right)}{e^{\annalpha y_i}} \stackrel{\eqref{IndAsymp}}{\le} \frac{1+\annerror{\anntime}\anntime\left(1+\eps\right)^\anntime\annalpha}{e^{n^{(i-1/2)c}}/2}\le e^{-n^{ic/3}}
\]
for $n$ sufficiently large. Hence
\begin{align*}
 \sum_{i=1}^\infty q_{y_i} y_{i+1} & \le \frac{(1+\eps)^\anntime}{2\eps} \sum_{i=1}^\infty e^{-n^{ic/3}}n^{(i+1)c}\\
& \le \eps^{-1}n^{j-1+\delta+c} e^{-n^{c/4}}\\
& = o(e^{-n^{c/5}}) .
\end{align*}
We have thus shown that in~\eqref{eq:markovtype} the contribution to $\EE (\lbrgensize{\anntime})$ from the upper tail is negligible. Now (\ref{eq:markovtype}) tells us that
\begin{align*}
q_z y_1 - q_z z \ge \EE (\lbrgensize{x}) - z +  o(e^{-n^{c/5}})
\end{align*}
and therefore
\begin{align*}
q_z & \ge \frac{(1+\eps)^\anntime - n^{j-1+\delta} +  o(e^{-n^{c/5}})}{(1+\eps)^\anntime n^c/(2\eps) -n^{j-1+\delta}}\\
& \ge \frac{\eps(1+\eps)^\anntime (1+o(1))}{2n^c (1+\eps)^\anntime (1+o(1)}\\
& \ge \eps /n^c,
\end{align*}
where we used the fact that $ n^{j-1+\delta}= o((1+\eps)^\anntime)$.
This completes the proof of Lemma~\ref{lem:getbig}.
\end{proof}

We now show how to deduce Lemma~\ref{lem:smallstart} from Lemma~\ref{lem:getbig}. We first consider the case $\ell=0$. In this case we aim to prove that with very high probability $e\left(\compone\left(\starttime{j-1}\right)\right) = o(\lcompconst n^{j})$.

We first observe that if the statement does not hold, then we certainly have $e\left(\compone\left(\starttime{j-1}\right)\right) \ge \lcompconst n^j /\omega$, for any $\omega\rightarrow \infty$, but each generation up to time $\starttime{j-1}$ only has size at most $n^{j-1+\delta}$.

We now consider two cases for possibilities of how the desired conclusion might fail, and show that each of these is very unlikely.

\noindent \textbf{Case 1:} There is a generation $i$ of size at least $n^{1/2-\delta/2+\annconst}$.

In this case, using $\lbr$ as a lower coupling for the search process, we begin $y=n^{1/2-\delta/2+\annconst}$ independent new processes at time $i$. By Lemma~\ref{lem:getbig}, each of these has a probability of at least $\eps /n^\annconst$ of reaching size at least $n^{j-1+\delta}$ within $\anntime$ steps. The probability that $\starttime{j-1}>i+\anntime$ is therefore at most
\begin{align*}
(1-\eps/n^\annconst)^y & \le \exp (-\eps y/n^\annconst)\\
& \le \exp (-n^{-1/3} n^{1/2-\delta/2 + \annconst}/n^\annconst)\\
& = \exp (-n^{\delta/2}).
\end{align*}
We may therefore assume that Case 1 does not occur for $i \le \starttime{j-1}-\anntime$.

\noindent \textbf{Case 2:} $\starttime{j-1}\ge (\anntime+1)n^{1/2-\delta/2 + \annconst}$.

To show that this is unlikely we once again aim to consider $y$ independent processes, but in order to do this we take one individual from each of the generations $i(\anntime+1)$, for $i=0,1,2,\ldots, n^{1/2-\delta/2 +\annconst}-1$ and consider the lower coupling $\lbr$ for the search process. Technically these are not independent processes, since one may be a subprocess of the other, but for this reason we consider $\lbr (\anntime)$, the process which is cut off after $\anntime$ steps. These processes are now independent, and the same calculation as above shows that the probability that none of them become large enough after $\anntime$ steps is very small. We may therefore assume that Case 2 does not occur.

However, if neither case occurs, then we have (assuming $j\ge 2$)
\begin{align*}
e(G') & \le (\anntime+1)n^{1/2-\delta/2+\annconst}n^{1/2-\delta/2+\annconst} + \anntime n^{j-1+\delta}\\
& \le 2\anntime (n^{1-\delta+2\annconst} + n^{j-1+\delta})\\
& \le n^{j-1 + 2\delta}/\eps\\
& = o(\lcompconst n^j)
\end{align*}
as required.

This shows that Lemma~\ref{lem:smallstart} holds for $\ell=0$. However, now we know that up to time $\starttime{j-1}$ we have found few edges, which intuitively should mean that we are early in the process, and that all maximum degrees will be small.
More precisely, $e\left(\compone\left(\starttime{j-1}\right)\right)\le \lcompconst n^j /\omega$ and therefore by Property~\eqref{prop:edges}, the number of queries we have made in the search process is at most
$$\tfrac{\lcompconst n^j }{\omega (1-\eo)p_0}  \le \lcompconst n^k/\sqrt{\omega}.$$
Therefore applying Property~\eqref{prop:crudemaxdeg} with $\alpha = \lcompconst/\sqrt{\omega}$, we have that
$$\Delta_\ell\left(\compone\left(\starttime{j-1}\right)\right) = O(\lcompconst n^{j-\ell}/\sqrt{\omega})=o(\lcompconst n^{j-\ell})$$
for ever $\ell = 1,\ldots,j-1$ as required. This completes the proof of Lemma~\ref{lem:smallstart}.\hfill \qed

\section{Concluding Remarks}\label{sec:conclusion}

\subsection{Hypertrees}

In this paper, we used the fact that Lemma~\ref{lem:crudemaxdeg} can be applied to the search process $\ubfs$ to show that $\lbr \prec \ubfs \prec \ubr$. However, this also applies to the following variant of the search process: We define $\lbfs$ to be the corresponding breadth-first search process in which a $k$-set may only be queried if it contains $\cconst$ neutral $j$-sets (as opposed to at least one for $\ubfs$).

$\lbfs$ is a search algorithm specifically looking for a tree, where we define a tree to be a component with $e$ edges and $\cconst e +1$ $j$-sets. Note that this algorithm will not necessarily reveal all of a component; however all the arguments involving $\ubfs$ which we use in this paper also hold for $\lbfs$. Thus we deduce that the number of $j$-sets contained in large trees is approximately $\tfrac{2\eps}{\cconst}\tbinom{n}{j}$, and indeed these (rooted) trees are smooth in the sense of the smooth boundary lemma.

\vspace{0.7cm}

Since random graphs have been so extensively studied, various further questions immediately suggest themselves, regarding whether we can also prove similar results for hypergraphs.

\subsection{The critical window}

Theorem~\ref{thm:main} contains two conditions which give lower bounds on $\eps$, namely $\eps^3 n^j \rightarrow \infty$ and $\eps^2 n^{1-2\delta}\rightarrow \infty$. The natural conjecture is that only the first of these should be genuinely necessary, and thus the result should hold provided $\eps \gg n^{-j/3}$. Note that for $\eps = \Theta(n^{-j/3})$, the bounds from the super-critical case ($\Theta (\eps n^j)$) and the sub-critical case ($O(\eps^{-2} \log n)$) match up to the $\log n$ term, suggesting that we have a smooth transition.

However, our proof method requires the additional second condition, which takes over when $j\ge 2$. Thus for these cases, our range of $\eps$ is probably not best possible.

This additional condition arises in the proof because we need the smooth boundary lemma. The boundary has size up to $\lcompconst^2 n^j=o(\eps^2 n^j)$, and the typical degree within the boundary of an $\ell$-set is $o(\eps^2 n^{j-\ell})$. If we are to show that these degrees are concentrated, we need this typical degree to be large, which in the case of $\ell=j-1$ means $\eps^2 n$ must be large, thus leading to our condition on $\eps$.

If we were to attempt to do away with this condition, we would need to have some control over degrees which may be very small. Presumably we would need to determine more precisely what the probability distribution of such degrees is.

\subsection{The distribution of the size of the giant component}

The size of the giant component shortly after the phase transition is a random variable whose mean we have (asymptotically) determined in this paper, and we have further shown that it is concentrated around its mean. However, we have not proved what the actual distribution of this random variable is.

The most likely candidate would be a normal distribution as is the case for graphs (as proved by Pittel and Wormald~\cite{PW05} and by Luczak and \L uczak~\cite{LL06}). More generally for the $1$-component in $k$-uniform hypergraphs, analogous results were proved for various ranges of $\eps$ by Karo\'nski and \L uczak~\cite{KarLuc02} and by Behrisch, Coja-Oghlan and Kang~\cite{BCOK14}, and for the whole of the supercritical regime ($\eps \gg n^{-1/3}$) by Bollob\'as and Riordan~\cite{BR12}. For these cases, central and local limit theorems are known. It would be interesting to prove similar results for the size of the largest $j$-component.

\subsection{The structure of the components}

For graphs, it is well known that shortly after phase transition, all components are whp trees or at most unicyclic. It would be interesting to know if something similar holds in hypergraphs. (A natural generalisation of unicyclic would be the most tree-like connected non-tree structure, i.e. a $j$-component with $e$ edges and $\cconst e$ $j$-sets. In this case, when discovering the component, exactly one $j$-set would be seen twice.) It follows very easily from the results of this paper and our remarks about $\lbfs$ that for \emph{almost} every $j$-set, the component containing it is a tree. Once again, the case $j=1$ has previously been studied (see~\cite{KarLuc02},~\cite{BCOK14}). 

\subsection{$\ell$-Cores}
The study of cores in random graphs was initiated by Bollob\'as. The $\ell$-core is the (unique) largest subgraph of minimum degree at least $\ell$.

For hypergraphs, the degree has many possible generalisations just as connectivity does. For each of these, we may then define the $\ell$-core and consider its properties. For vertex-degree, Molloy~\cite{Mo05} determined the critical threshold for the existence of a non-empty $\ell$-core and its asymptotic size, but in general this is still an open question.

These questions and many more suggest themselves as the first steps in the developing field of random hypergraph theory.

\appendix

\section{Proof of Lemma~\ref{lem:expansionbound}}\label{app:proofs}

Recall that we proved only the upper bound for the first half of Lemma~\ref{lem:expansionbound}. For the lower bound, we note that since we are in the range where $\lbr$ will give us a lower bound, $\jgensize{i+1}$ stochastically dominates a random variable $Z_{i+1}$ with the distribution of $\cconst\cdot\Bi \big(x_i (1-\elbr)\binom{n}{k-j},p\big)$. Thus we have, again by the Chernoff bound (Theorem~\ref{thm:chernoff}),
\begin{align*}
\Pr \left(\jgensize{i+1}\le (1-2\elbr)(1+\eps)x_i\right) & \le \Pr\left(Z_{i+1}\le (1-2\elbr)(1+\eps)x_i\right)\\
& \le \Pr \left(\tfrac{Z_{i+1}}{\cconst}\le (1-\elbr)\EE\left(\tfrac{Z_{i+1}}{\cconst}\right)\right)\\
& \le  \exp \left(\frac{-(\elbr)^2 \EE(Z_{i+1})}{2\cconst}\right)\\
& = \exp \left(-\Theta \left(\elbr^2 x_i\right)\right)\\
& \le \expprob
\end{align*}
provided that $x_i \ge n^{1-\delta}$.

For the second half, the calculation is very similar. $\jgensize{i+1}$ is dominated by a random variable $Y_{i+1}$ with distribution $\cconst\cdot\Bi \big(n^{\delta} \binom{n}{k-j},p\big)$, and so
\begin{align*}
\Pr \left(\jgensize{i+1}\ge 2n^{\delta}\right) & \le \Pr\left(Y_{i+1}\ge 2n^{\delta}\right)\\
& \le  \exp \left(\frac{-(1-\eps)^2 \EE(Y_{i+1})}{3\cconst}\right)\\
& = \exp \left(-\Theta \left(n^{\delta}\right)\right)\\
& \le \expprob. 
\end{align*}

\vspace{-0.6cm}\qed \vspace{0.2cm}

\section{Branching Processes}\label{app:branchings}

Recall that we aim to calculate the asymptotic value of the survival probability $\varrho$ of the branching process $\ubr$, in which the number of children has distribution $\cconst \cdot \Bi \big(\binom{n}{k-j},p\big)$, where $p=(1+\eps)p_0$.
For technical reasons, it is slightly easier first to calculate the probability $\varrho_0$ that at least one of $\cconst$ independent branching processes (each an instance of $\ubr$) survives. In this case the number of children in the first generation has distribution $\Bi \big(\cconst \binom{n}{k-j},p\big)$.

By standard results for branching processes (see~\cite{Harris63}) we know that $\varrho>0$. We note that the processes all die out if and only if all subprocesses starting at children in the first generation die out. Making a case distinction on the number of such children, we obtain
\begin{align*}
1-\varrho_0 &= \sum_{i=0}^{\cconst\binom{n}{k-j}} \Pr \left(\Bi \left(\cconst\binom{n}{k-j},p\right)=i\right) (1-\varrho_0)^{i}\\
& = \sum_{i=0}^{\cconst\binom{n}{k-j}} \binom{\cconst \binom{n}{k-j}}{i} p^i (1-p)^{\cconst\binom{n}{k-j}-i} (1-\varrho_0)^i\\
& = \left( p(1-\varrho_0) + 1-p\right)^{\cconst \binom{n}{k-j}}\\
&=\left( 1-p\varrho_0 \right)^{\cconst \binom{n}{k-j}}\\
& = \left(1-(1+\eps)p_0\varrho_0\right)^{p_0^{-1}}.
\end{align*}
Solving this equation for $\eps$ yields
\begin{equation}\label{survivalEquationSolvedForVareps}
\eps=\frac{1-\left(1-\varrho_0\right)^{p_0}}{\varrho_0p_0}-1=\frac{f(\varrho_0)}{\varrho_0p_0}\, ,
\end{equation}
for 
\begin{align*}
f(\varrho_0)&=1-\varrho_0p_0-\left(1-\varrho_0\right)^{p_0}\\
& = \frac{p_0(1-p_0)}{2}\varrho_0^2 + \frac{p_0(1-p_0)(2-p_0)}{6}\varrho_0^3 + \ldots\, . 
\end{align*}
Since this expression has only non-negative coefficients, \eqref{survivalEquationSolvedForVareps} implies that
$$\eps\geq \frac{\frac{p_0(1-p_0)}{2}\varrho_0^2}{\varrho_0p_0}=\frac{\varrho_0}{2}-O(p_0\varrho_0)$$
 and hence $\varrho_0=o(1)$ since $\eps=o(1)$ and $p_0=o(1)$. Consequently we derive the asymptotic estimate
 \begin{equation}\label{survivalSnowball}
 \eps=\frac{\varrho_0}{2}+O\left(\varrho_0p_0+\varrho_0^2\right)\sim \frac{\varrho_0}{2}
 \end{equation}
  from \eqref{survivalEquationSolvedForVareps}. Since $1-\varrho_0=(1-\varrho)^\cconst$, we have 
  \begin{equation}
  \varrho = 1-\left(1-\varrho_0\right)^{\frac{1}{\cconst}}=\frac{\varrho_0}{\cconst}+O\left(\varrho_0^2\right)\stackrel{\eqref{survivalSnowball}}{\sim} \frac{2\eps}{\cconst}\, .\label{survival}
  \end{equation}

\begin{remark}\label{remarkSurvivalLbr}
Note that an almost identical calculation shows that the survival probability $\varrho_*$ of the lower coupling process $\lbr$ also satisfies 
\begin{equation}
\varrho_*\sim \frac{2\eps}{\binom{k}{j}-1}\, .
\label{survivalSub}
\end{equation}
This is because the expected number of children of each individual is $(1-\elbr)(1+\eps) = 1+\eps + o(\eps)$.
\end{remark}

Recall that $\dualbr$ is the dual process, i.e. the process $\ubr$ conditioned on $\death$, the event that $\ubr$ dies. We consider the children of an individual as being grouped into $\emph{litters}$, each containing $\cconst$ children, where the number of litters has distribution $\Bi \big(\binom{n}{k-j},p\big)$ in $\ubr$. We want to show that $\dualbr$ is a branching process in which each individual has a number of litters of children which has binomial distribution.
 
To analyse the dual process we consider the change of the probability of $\death$ subject to the presence $\event{e}$ of a litter of children of an individual $J$. We denote by $\gen{J}$ the set of individuals in the same generation (of $\ubr$) as $J$ and calculate the probabilities $\Pr\left(\death\cond\event{e}\right)$ and $\Pr\left(\death\cond\neg\event{e}\right)$ by conditioning on the number of litters $s(\gen{J})$ of children of individuals in $\gen{J}$ and obtain
\begin{align*}
\Pr \left(\death\cond\event{e}\right) &= \sum_{i=0}^{M-1}\Pr(s\left(\gen{J})=i+1\cond \event{e}\right)(1-\varrho)^{(i+1)\cconst}\\
& = \sum_{i=0}^{M-1} \Pr \big(\Bi (M-1,p)=i\big)(1-\varrho)^{(i+1)\cconst}\, ,\\
\Pr \left(\death\cond\neg\event{e}\right) & = \sum_{i=0}^{M-1}\Pr\left(s(\gen{J})=i\cond \neg \event{e}\right)(1-\varrho)^{i\cconst}\\
 & = \sum_{i=0}^{M-1} \Pr \big(\Bi (M-1,p)=i\big)(1-\varrho)^{i\cconst}
\end{align*}
and so
\begin{align}\label{changeDeathProbability}
\frac{\Pr\left(\death\cond\event{e}\right)}{\Pr\left(\death\cond\neg\event{e}\right)}=\left(1-\varrho\right)^{\cconst}.
\end{align}
Therefore we get
\begin{align}
\Pr\left(\event{e}\cond \death\right)
&=\frac{\Pr\left(\death\cond\event{e}\right)\Pr\left(\event{e}\right)}{\Pr\left(\death\cond\event{e}\right)\Pr\left(\event{e}\right)+\Pr\left(\death\cond\neg\event{e}\right)\Pr\left(\neg\event{e}\right)}\nonumber\\
&=\frac{\frac{\Pr\left(\death\cond\event{e}\right)}{\Pr\left(\death\cond\neg\event{e}\right)}\cdot\Pr\left(\event{e}\right)}{\frac{\Pr\left(\death\cond\event{e}\right)}{\Pr\left(\death\cond\neg\event{e}\right)}\cdot\Pr\left(\event{e}\right)+\Pr\left(\neg\event{e}\right)}\nonumber\\
&\stackrel{\eqref{changeDeathProbability}}{=}\frac{\left(1-\varrho\right)^{\cconst}p}{1-p\left(1-\left(1-\varrho\right)^{\cconst}\right)}\nonumber\, ,
\end{align}
and note in particular that this probability is independent of the choice of $e$ and $J$, hence we denote it by $\pdualbr$. Moreover we obtain the estimate 
\begin{align}
\pdualbr& = \frac{(1-\varrho)^{\cconst}p}{1-O(p\varrho)}\nonumber\\
&=p\left(1-\varrho\right)^{\cconst}+O(p^2\varrho)\nonumber\\
&\stackrel{\eqref{survival}}{=}(1-\eps+o(\eps))p_0\,\nonumber .
\end{align} 
Furthermore a very similar calculation shows that conditioned on $\death$ the presence of $e$ is still independent of all other edges. Hence the dual process $\dualbr$ is a branching process whose offspring distribution is given by
$$\cconst \cdot \Bi\left(\binom{n}{k-j},\pdualbr\right).$$

\ 

\bibliographystyle{amsplain}
\bibliography{Bibliography}

\ 

\end{document}